\documentclass[11pt]{article}
\usepackage{graphicx,psfrag}
\usepackage{amssymb,amsfonts,amsthm}
\usepackage{amsmath,amsthm,amssymb}
\usepackage{amsfonts}
\usepackage{relsize}
\usepackage[T1]{fontenc}
\usepackage[utf8]{inputenc}
\usepackage{graphicx}
\usepackage{color}
\usepackage{subcaption}

\newcommand{\Cl}{\mathrm{Cl}}

\newcommand{\DG}{\mathrm{DG}}

\newcommand{\dd}{\mathcal{D}}

\newcommand{\la}{\langle}
\newcommand{\ra}{\rangle}

\newcommand{\F}{\mathcal{F}}

\newtheorem{theorem}{Theorem}[section]

\newtheorem{cor}[theorem]{Corollary}

\theoremstyle{definition}
\newtheorem{definition}[theorem]{Definition}

\newtheorem{prob}[theorem]{Problem}

\newcommand{\kk}{\mathcal{K}}
\newcommand{\PP}{\mathcal {P}}

\newcommand{\be}{\begin{equation}}
\newcommand{\ee}{\end{equation}}

\newcommand {\N}{\mathbb{N}} %% positive integers
\newcommand {\Z}{\mathbb{Z}}            %% integers
\newcommand {\R}{\mathbb{R}} %% reals
\newcommand {\iv}{^{-1}}

\numberwithin{equation}{section}

\newcounter{AbcT}

\newcommand{\La}{\mathcal{L}}

\newcommand{\nc}{\newcommand}
\nc{\meet}{\wedge}
\nc{\op}{\operatorname}\nc{\FP}{\op{FP}}\nc{\FS}{\op{FS}}\nc{\FPhat}{\widehat{\op{FP}}}

\newtheorem {Theorem}    {Theorem}[section]

\newtheorem {Problem}    [Theorem]{Problem}

\newtheorem {Lemma}      [Theorem]    {Lemma}
\newtheorem {Corollary}   [Theorem] {Corollary}
\newtheorem {Proposition}[Theorem]    {Proposition}

\theoremstyle{remark}
\newtheorem {Remark}		 [Theorem]    {\bf{Remark}}

\newtheorem {Definition} [Theorem]    {\bf{Definition}}

\newcommand{\topp}{{\bf top}}

\newcommand{\bott}{{\bf bot}}
\newcommand{\1}{\mathbf{1}}
\newcommand{\arr}{\overrightarrow}

\begin{document}

\title{On subgroups of R. Thompson's group $F$}
\author{Gili Golan, Mark Sapir\thanks{The research was partially supported by the NSF grant DMS-1500180. The paper was written while the second author was visiting the Max Planck Institute for Mathematics in Bonn.}}%
%\thanks{The research was supported in part by the NSF grants DMS 1418506, DMS	1318716 and by the BSF grant 2010295.}}
\maketitle

\abstract{We provide two ways to show that the R. Thompson group $F$ has maximal subgroups of infinite index which do not fix any number in the unit interval under the natural action of $F$ on $(0,1)$, thus solving a problem by D. Savchuk. The first way employs Jones' subgroup of the R. Thompson group $F$ and leads to an explicit finitely generated example. The second way employs directed 2-complexes and 2-dimensional analogs of Stallings' core graphs, and gives many implicit examples.  We also show that $F$ has a decreasing sequence of finitely generated subgroups $F>H_1>H_2>...$ such that $\cap H_i=\{1\}$ and for every $i$ there exist only finitely many subgroups of $F$ containing $H_i$. }

\section{Introduction}

\subsection{The stories about $F$}

Recall that the R. Thompson group $F$ consists of all piecewise linear increasing homeomorphisms of the unit interval $[0,1]$ with a finite number of linear segments with slopes of the form $2^n, n\in \Z$, and endpoints of the form $\frac{a}{2^n}, a, n\in \N\cup \{0\}$. The group $F$ has a presentation with two generators and two defining relations \cite{CFP, Sbook} (see below).

The group $F$ is one of the most mysterious objects in group theory. For example,  almost every 6 months a new ``proof'' of amenability or non-amenability of $F$ appears. The reason why all these proofs are wrong is that $F$ is very counter-intuitive. Statements which are ``obviously true'' turn out to be wrong.

This can also be illustrated by the story of finding  the Dehn function of $F$. First it was declared exponential because a sequence of loops in the Cayley graph of $F$ was found with exponential fillings, and it was ``obviously impossible'' to find fillings of these loops with smaller area. Then  it was proved to be subexponential and the most probable conjecture (for the same reason as before) was that it is something like $2^{\log^2 n}$. Then a polynomial of degree 5 upper bound was found and it was conjectured that this bound is optimal. Then a few lower bounds were found until Guba proved that $F$ has quadratic Dehn function, the smallest possible Dehn function of a non-hyperbolic finitely presented group (see the references in \cite[Chapter 5]{Sbook}).

\subsection{Subgroups of $F$}

Subgroups of $F$ have been extensively studied (see, for instance, \cite{Burillo,CFP, GuSa99,Brin,Bleak1,Cleary,Sav1,Sav,Bleak,GS,WC}). For example, it is known that the derived subgroup $[F,F]$ of $F$ is simple, and every finite index subgroup of $F$ contains $[F,F]$ \cite{CFP}. Finite index subgroups of $F$ are described in \cite{Bleak} and all solvable subgroups of $F$ are described in \cite{Bleak1}. It is also known \cite{GuSa99} that the wreath products $F\wr Z$ embeds into $F$ (even without distortion \cite{WC}). One of the most interesting and counter-intuitive results about subgroups of $F$ is that in a certain natural probabilistic model on the set of all finitely generated subgroups of $F$, every finitely generated nontrivial subgroup appears with a positive probability \cite{Cleary}. The survey in this paragraph is far from complete. We just wanted to give a rough description of the area.

In \cite{Sav1, Sav} Dmytro Savchuk studied subgroups $H_U$ of the group $F$ which are the stabilizers of finite sets of real numbers $U\subset(0,1)$. He proved that if $U$ consists of one number, then $H_U$ is a maximal subgroup of $F$. He also showed that the Schreier graphs of the subgroups $H_U$ are amenable.  He asked \cite[Problem 1.5]{Sav} whether every maximal subgroup of infinite index in $F$ is of the form $H_{\{\alpha\}}$, that is, fixes a number from $(0,1)$.\footnote{In the formulation of \cite[Problem 1.5]{Sav} the words ``of infinite index'' are missing. D. Savchuk informed us that this is a misprint. Obviously none of the finite index maximal subgroups of $F$ fixes a point in $(0,1)$.}

In this paper, we are going to answer that question by establishing a maximal subgroup of infinite index in $F$ that does not fix any real number in $(0,1)$ (see Section \ref{s:sav}). In fact we give two solutions of this problem. One, explicit solution, gives an example as the preimage of the subgroup $\arr F$, first described by Vaughan Jones \cite{Jones} and then studied by us in \cite{GS}, under a certain injective endomorphism of $F$ (Theorem \ref{t:11}). As a by-product of the proof, we give a positive answer to a question asked by Saul Schleimer: we prove that $\arr F$ is maximal in a certain subgroup of index $2$ of $F$ and, moreover,  there are exactly three subgroups of $F$ containing $\arr F$ (that fact, Corollary \ref{cor}, is also very counter-intuitive).

Another, implicit, solution gives an example (in fact many examples) of maximal subgroups containing certain proper finitely generated subgroups $H$ of $F$ which do not fix any point in $(0,1)$ and satisfy $H[F,F]=F$. The most difficult part of that solution is to prove that $H$ is indeed a proper subgroup of $F$.
% Indeed, $H$ is not in a finite index proper subgroup of $F$ and does not fix a point on the open unit interval - the only previously known indicators that a subgroup of $F$ is proper.
  Note that the solvability of the membership problem for subgroups in $F$ is a very interesting open problem (it is mentioned in \cite{GuSa99}) and the generation problem is an important open case of the membership problem (we need to check if the generators of $F$ are in the subgroup). To prove that $H$ is proper we employ diagram groups of directed 2-complexes \cite{GSdc} and a 2-dimensional analog of Stallings foldings.

We also establish a nice property of Savchuk's subgroups $H_U$.

Let us call a subgroup $H$ of a group $G$ of {\em quasi-finite index} if the interval $[H,G]$ in the lattice of subgroups  of $G$ is finite, that is there are only finitely many subgroups of $G$ that are bigger than $H$. For example, every maximal subgroup and every subgroup of finite index are of quasi-finite index. We say that a finitely generated group $G$ is {\em quasi-residually finite} if it contains a decreasing sequence of finitely generated subgroups: $G> G_1 > G_2 >...$ such that each $G_i$ is of quasi-finite index in $G$ and $\cap G_i=\{1\}$. Clearly every residually finite group is  quasi-residually finite. Since the derived subgroup of the R. Thompson group $F$ is  infinite and simple, $F$  is not residually finite. Nevertheless, we prove that if $U$ is a finite set of finite binary fractions, then $H_U$ is of quasi-finite index in $F$ (this is a natural generalization of a result from \cite{Sav1}), and hence $F$ is quasi-residually finite (Theorem \ref{t:2}).

\section{Preliminaries on $F$}\label{sec:pre}

\subsection{Generators and normal forms}\label{ss:nf}

The group $F$ is generated by two functions $x_0$ and $x_1$ defined as follows \cite{CFP}

	\[
   x_0(t) =
  \begin{cases}
   2t &  \hbox{ if }  0\le t\le \frac{1}{4} \\
   t+\frac14       & \hbox{ if } \frac14\le t\le \frac12 \\
   \frac{t}{2}+\frac12       & \hbox{ if } \frac12\le t\le 1
  \end{cases} 	\qquad	
   x_1(t) =
  \begin{cases}
   t &  \hbox{ if } 0\le t\le \frac12 \\
   2t-\frac12       & \hbox{ if } \frac12\le t\le \frac{5}{8} \\
   t+\frac18       & \hbox{ if } \frac{5}{8}\le t\le \frac34 \\
   \frac{t}{2}+\frac12       & \hbox{ if } \frac34\le t\le 1 	
  \end{cases}
\]

%One can see that $x_1$ is the identity on $[0,\frac12]$ and a shrank by the factor of 2 copy of $x_0$ on $[\frac12, 1]$.
The composition in $F$ is from left to right.

Every element of $F$ is determined by its action on the set of finite binary fractions which can be represented as $.\alpha$ where $\alpha$ is an infinite word in $\{0,1\}$ where all but finitely many letters are $0$. For each element $g\in F$ there exists a finite collection of pairs of words $(u_i,v_i)$ in the alphabet $\{0,1\}$ such that every infinite word in $\{0,1\}$ starts with exactly one of the $u_i$'s. The action of $F$ on a number $.\alpha$ is the following: if $\alpha$ starts with $u_i$, we replace $u_i$ by $v_i$. For example, $x_0$ and $x_1$  are the following functions:

\[
   x_0(t) =
  \begin{cases}
   .0\alpha &  \hbox{ if }  t=.00\alpha \\
    .10\alpha       & \hbox{ if } t=.01\alpha\\
   .11\alpha       & \hbox{ if } t=.1\alpha\
  \end{cases} 	\qquad	
   x_1(t) =
  \begin{cases}
   .0\alpha &  \hbox{ if } t=.0\alpha\\
   .10\alpha  &   \hbox{ if } t=.100\alpha\\
   .110\alpha            &  \hbox{ if } t=.101\alpha\\
   .111\alpha                      & \hbox{ if } t=.11\alpha\
  \end{cases}
\]

For the generators $x_0,x_1$ defined above, the group $F$ has the following finite presentation \cite{CFP}
$$F=\la x_0,x_1\mid [x_0x_1^{-1},x_1^{x_0}]=1,[x_0x_1^{-1},x_1^{x_0^2}]=1\ra.$$

Often, it is more convenient to consider an infinite presentation of $F$. For $i\ge 1$, let $x_{i+1}=x_1^{x_0^i}$ where $a^b$ denotes $b\iv ab$. In these generators, the group $F$ has the following presentation \cite{CFP}
$$\la x_i, i\ge 0\mid x_i^{x_j}=x_{i+1} \hbox{ for every}\ j<i\ra.$$

There is a natural notion of normal forms of elements of $F$ related to this infinite presentation (see \cite{CFP}).
A word $w$ in the alphabet $\{x_0,x_1,x_2,\dots\}^{\pm 1}$ is said to be a \emph{normal form} if
\begin{equation}\label{NormForm}
w= x_{i_1}^{s_1}\ldots x_{i_m}^{s_m}x_{j_n}^{-t_n}\ldots x_{j_1}^{-t_1},
\end{equation}
where
\begin{enumerate}
\item[(1)] $i_1\le\cdots\le i_m\ne j_n\ge\cdots\ge j_1$;
$s_1,\ldots,s_m,t_1,\ldots t_n>0$, and
\item[(2)] if $x_i$ and $x_i\iv$ occur in
(\ref{NormForm}) for some $i\ge0$ then either $x_{i+1}$ or $x_{i+1}\iv$
also occurs in~(\ref{NormForm}).
\end{enumerate}
Every element of $F$ is uniquely represented by a normal form $w$. Moreover, the normal form of an element $w\in F$ is the shortest possible representation of the element as a word in $\{x_0,x_1,x_2,\dots\}^{\pm 1}$.
A word $w$ of the form (\ref{NormForm}) which satisfies condition (1) but not necessarily condition (2) above, is called a \emph{semi-normal form}.

Given a word $w$ in the alphabet $\{x_0,x_1,x_2,\dots\}^{\pm 1}$, there is a simple algorithm for finding the corresponding normal form of $w$ (see \cite{SU}).
First, one uses the following rewriting system $\Sigma$. %, until no rule is applicable to it.
\begin{enumerate}
\item[($\sigma_1$)] $x_ix_j\rightarrow x_jx_{i+1}$ for $i>j$.
\item[($\sigma_2$)] $x_i^{-1}x_j\rightarrow x_jx_{i+1}^{-1}$ for $i>j$.
\item[($\sigma_3$)] $x_i^{-1}x_j\rightarrow x_{j+1}x_i^{-1}$ for $j>i$.
\item[($\sigma_4$)] $x_i^{-1}x_j^{-1}\rightarrow x_{j+1}^{-1}x_i^{-1}$ for $j>i$.
\item[($\sigma_5$)] $x_i^{-1}x_i\rightarrow 1$.
\item[($\sigma_6$)] $x_ix_i^{-1}\rightarrow 1$.
\end{enumerate}
The rewriting system is terminating and confluent (for the terminology see \cite{Sbook}), hence every word $w$ can be transformed by applying a sequence of rewriting rules to a unique word $w'$ for which no rewriting rule can be applied. That word $w'$ does not contain subwords that are equal to the left hand sides of the rewriting rules, hence it is a semi-normal form (it is easy to verify). It is obvious that the word $w'$ is equal to $w$ in $F$ (see \cite[Lemma 1]{SU}).

The method of getting a normal form from a semi-normal form will not be important to us. Suffice it to know, that if $w'$ is a semi-normal form which does not satisfy condition (2) above, then the length of the corresponding normal form will be at most $|w'|-2$ (see  \cite[Lemma 5]{SU}).
Below, we will not distinguish between an element of $F$ and its normal form. %That is, we will always assume that elements are given in their normal form.

Every normal form $w$ is a product $pq\iv$ where $p, q$ are positive words. The word $p$ is called the \emph{positive part} of $w$ and $q\iv$ is the \emph{negative part} of $w$. Every element whose normal form is a positive word (i.e., its negative part is empty), is called a \emph{positive element}. The set of positive elements of $F$ forms a submonoid of $F$ because the rewriting rules do not increase the number of negative letters in a word.

\subsection{Jones' subgroup}

Vaughan Jones \cite{Jones} showed that, similar to braids, elements of $F$ can be used to construct all links. He also showed  that a certain subgroup $\arr F$ of $F$ can be used to construct all oriented links. In \cite{GS} we answered several questions of Jones about $\arr F$. In particular we proved that $\arr F$ is isomorphic to the group $F_3$ introduced by K. Brown \cite{B}: it is a version of $F$ where all slopes of the piecewise-linear functions are powers of $3$ and break points of the derivative are finite 3-adic fractions. We also showed that $\arr F$ is generated in $F$ by elements $x_0x_1, x_1x_2, x_2x_3$.
By conjugating $x_1x_2$ and $x_2x_3$ by $(x_0x_1)^n$ one gets that for all $n\ge 0$, $x_nx_{n+1}\in\arr F$.

\subsection{Directed complexes and  diagram groups} \label{ss:dc}

The material from this subsection will be used only in Sections \ref{ss:imp} and \ref{53}. Here we basically follow \cite{GSdc}.

\begin{Definition}
\label{dirc}
{\rm
For every directed graph $\Gamma$ let $\PP$ be the set of all (directed)
paths in $\Gamma$, including the empty paths. A {\em directed $2$-complex\/}
is a directed graph $\Gamma$ equipped with a set $\F$ (called the {\em set
of $2$-cells\/}), and three maps $\topp{\cdot}\colon\F\to\PP$,
$\bott{\cdot}\colon\F\to\PP$, and $^{-1}\colon\F\to\F$ called {\em top\/},
{\em bottom\/}, and {\em inverse\/} such that
\begin{itemize}
\item for every $f\in\F$, the paths $\topp(f)$ and $\bott(f)$ are non-empty
and have common initial vertices and common terminal vertices,
\item $^{-1}$ is an involution without fixed points, and
$\topp(f^{-1})=\bott(f)$, $\bott(f^{-1})=\topp(f)$ for every $f\in\F$.
\end{itemize}
}
\end{Definition}

We shall often need an orientation on $\F$, that is, a subset
$\F^+\subseteq\F$ of {\em positive\/} $2$-cells, such that $\F$ is the
disjoint union of $\F^+$ and the set $\F^-=(\F^+)^{-1}$ (the latter is called
the set of {\em negative\/} $2$-cells).

%\end{document}

If $\kk$ is a directed $2$-complex, then paths on $\kk$ are called
{\em $1$-paths\/} (we are going to have $2$-paths later). The initial and
terminal vertex of a $1$-path $p$ are denoted by $\iota(p)$ and
$\tau(p)$, respectively. For every $2$-cell $f\in\F$, the vertices
$\iota(\topp(f))=\iota(\bott(f))$ and $\tau(\topp(f))=\tau(\bott(f))$ are
denoted $\iota(f)$ and $\tau(f)$, respectively.

We shall denote each cell $F$ by $\topp(f)\to \bott(f)$. And we can denote a directed 2-complex $\kk$ similar to a semigroup presentaton
$\la E\mid \topp(f)\to \bott(f), f\in F\ra$ where $E$ is the set of all edges of  $\kk$ (note that we ignore the vertices of $\kk$).

For example, the complex $\la x\mid x\to x^2\ra$ is the {\em Dunce hat\/}
obtained by identifying all edges in the triangle (Figure \ref{f:1})
according to their directions.  It has one vertex, one edge, and one
positive $2$-cell.

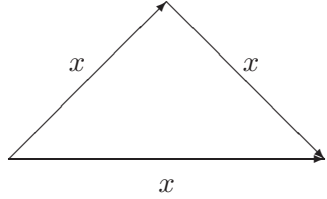
\begin{figure}[ht]
\begin{center}
\unitlength=0.7mm
\special{em:linewidth 0.4pt}
\linethickness{0.4pt}
\begin{picture}(61.00,36.00)
\put(1.00,6.00){\vector(1,1){30.00}}
\put(31.00,36.00){\vector(1,-1){30.00}}
\put(1.00,6.00){\vector(1,0){60.00}}
\put(31.00,1.00){\makebox(0,0)[cc]{$x$}}
\put(14.00,24.00){\makebox(0,0)[cc]{$x$}}
\put(47.00,24.00){\makebox(0,0)[cc]{$x$}}
\end{picture}

\caption{Dunce hat}
\label{f:1}
\end{center}
\end{figure}
%\addtocounter{ppp}{1}

With the directed 2-complex $\kk$, one can associate a category $\Pi(\kk)$
whose objects are directed $1$-paths, and morphisms are {\em $2$-paths\/}, i.\,e. sequences of replacements of $\topp(f)$ by
$\bott(f)$ in $1$-paths, $f\in\F$. More precisely, an {\em atomic $2$-path\/}
(an {\em elementary homotopy\/}) is a triple $(p,f,q)$, where $p$, $q$ are
$1$-paths in $\kk$, and $f\in\F$ such that $\tau(p)=\iota(f)$,
$\tau(f)=\iota(q)$. If $\delta$ is the atomic $2$-path $(p,f,q)$, then
$p\topp(f)q$ is denoted by $\topp(\delta)$, and $p\bott(f)q$ is denoted by
$\bott(\delta)$; these are called the {\em top\/} and the {\em bottom\/}
$1$-paths of the atomic $2$-path. Every {\em nontrivial\/} $2$-path $\delta$
on $\kk$ is a sequence of atomic paths $\delta_1$, \dots, $\delta_n$, where
$\bott(\delta_i)=\topp(\delta_{i+1})$ for every $1\le i<n$. In this case $n$
is called the {\em length\/} of the $2$-path $\delta$. The {\em top\/} and
the {\em bottom\/} $1$-paths of $\delta$, denoted by $\topp(\delta)$ and
$\bott(\delta)$, are $\topp(\delta_1)$ and $\bott(\delta_n)$, respectively.
 Every $1$-path $p$ is considered as a trivial
$2$-path with $\topp(p)=\bott(p)=p$. These are the identity morphisms in the
category $\Pi(\kk)$. The composition of $2$-paths $\delta$ and $\delta'$ is
called {\em concatenation\/} and is denoted $\delta\circ\delta'$.

With every atomic $2$-path $\delta=(p,f,q)$, where $\topp(f)=u$, $\bott(f)=v$
we associate the labeled plane graph $\Delta$ in Figure \ref{f:2}. An arc
labeled by a word $w$ is subdivided into $|w|$ edges. All edges are
oriented from left to right. The label of each oriented edge of the
graph is a symbol from the alphabet $E$, the set of edges in $\kk$. As a
plane graph, it has only one bounded face; we label it by the corresponding
cell $f$ of $\kk$. This plane graph $\Delta$ is called the diagram of $\delta$.
Such diagrams are called {\em atomic\/}. The leftmost and rightmost vertices
of $\Delta$ are denoted by $\iota(\Delta)$ and $\tau(\Delta)$, respectively.
The diagram $\Delta$ has two distinguished paths from $\iota(\Delta)$ to
$\tau(\Delta)$ that correspond to the top and bottom paths of $\delta$. Their
labels are $puq$ and $pvq$, respectively. These are called the top and the
bottom paths of $\Delta$ denoted by $\topp(\Delta)$ and $\bott(\Delta)$.

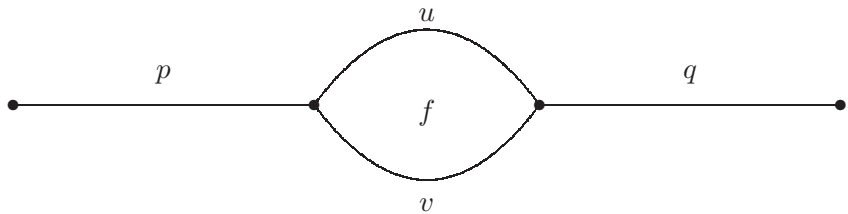
\begin{figure}[ht]
\begin{center}
\unitlength=1.00mm
\special{em:linewidth 0.4pt}
\linethickness{0.4pt}
\begin{picture}(110.66,26.33)
\put(0.00,13.33){\circle*{1.33}}
\put(40.00,13.33){\circle*{1.33}}
\put(70.00,13.33){\circle*{1.33}}
\put(110.00,13.33){\circle*{1.33}}
\bezier{200}(40.00,13.33)(55.00,33.33)(70.00,13.33)
\bezier{200}(40.00,13.33)(55.00,-6.67)(70.00,13.33)
\put(20.00,17.33){\makebox(0,0)[cc]{$p$}}
\put(55.00,25.33){\makebox(0,0)[cc]{$u$}}
\put(90.00,17.33){\makebox(0,0)[cc]{$q$}}
\put(55.00,12.33){\makebox(0,0)[cc]{$f$}}
\put(55.00,0.00){\makebox(0,0)[cc]{$v$}}
\put(0.33,13.33){\line(1,0){39.67}}
\put(70.00,13.33){\line(1,0){40.00}}
\end{picture}
\vspace{1ex}

\nopagebreak[4] %Figure \theppp.
\end{center}
\caption{An atomic diagram}
\label{f:2}
\end{figure}
%\addtocounter{ppp}{1}

The diagram corresponding to the trivial $2$-path $p$ is just an arc labeled
by $p$; it is called a {\em trivial diagram\/} and it is denoted by $\varepsilon(p)$.

Let $\delta=\delta_1\circ\delta_2\circ\cdots\circ\delta_n$ be a $2$-path
in $\kk$, where $\delta_1$, \dots, $\delta_n$ are atomic $2$-paths. Let
$\Delta_i$ be the atomic diagram corresponding to $\delta_i$. Then the
bottom path of $\Delta_i$ has the same label as the top path of $\Delta_{i+1}$
($1\le i<n$). Hence we can identify the bottom path of $\Delta_i$ with the
top path of $\Delta_{i+1}$ for all $1\le i<n$, and obtain a plane graph
$\Delta$, which is called the {\em diagram of the $2$-path\/} $\delta$.

Two diagrams are considered {\em equal\/} if they are isotopic as plane
graphs. The isotopy must take vertices to vertices, edges to edges, it must
also preserve labels of edges and   cells. Two $2$-paths are
called {\em isotopic\/} if the corresponding diagrams are equal.

Concatenation of $2$-paths corresponds to the concatenation of diagrams:
if the bottom path of $\Delta_1$ and the top path of $\Delta_2$ have the
same labels, we can identify them and obtain a new diagram
$\Delta_1\circ\Delta_2$.

Note that for any atomic $2$-path $\delta=(p,f,q)$ in $\kk$ one can naturally
define its {\em inverse\/} $2$-path $\delta^{-1}=(p,f^{-1},q)$. The inverses of
all $2$-paths and diagrams are defined naturally. The inverse diagram $\Delta^{-1}$
of $\Delta$ is obtained by taking the mirror image of $\Delta$ with respect to a
horizontal line, and replacing labels of cells by their inverses.

Let us identify in the category $\Pi(\kk)$ all isotopic $2$-paths and also identify
each $2$-path of the form $\delta'\delta\delta^{-1}\delta''$ with $\delta'\delta''$.
The quotient category is obviously a groupoid (i.\,e. a category with invertible
morphisms). It is denoted by $\dd(\kk)$ and is called the {\em diagram groupoid\/}
of $\kk$. Two $2$-paths are called {\em homotopic\/} if they correspond to the same
morphism in $\dd(\kk)$. For each $1$-path $p$ of $\kk$, the local group of $\dd(\kk)$
at $p$ (i.e., the group of homotopy classes of $2$-paths connecting $p$ with
itself) is called the {\em diagram group of the directed $2$-complex $\kk$ with base\/}
$p$ and is denoted by $\DG(\kk,p)$.

The following theorem is proved in \cite{GuSa97} (see also \cite{GSdc}).

\begin{Theorem} If $\kk$ is the Dunce hat (see Figure \ref{f:1}) and $p$ is the edge of it,
then $\DG(\kk,p)$ is isomorphic to the R. Thompson group $F$. The generators $x_0, x_1$ of $F$ are depicted in Figure \ref{f:xx} (all edges in the diagrams are labeled by $x$ and oriented from left to right).
\end{Theorem}

\begin{figure}[h!]
\begin{center}
\unitlength 1mm % = 2.845pt
\linethickness{0.4pt}
\ifx\plotpoint\undefined\newsavebox{\plotpoint}\fi % GNUPLOT compatibility
\begin{picture}(94.5,37)(0,0)
\put(39.5,24){\circle*{2}}
\put(29.5,24){\circle*{2}}
\put(19.5,24){\circle*{2}}
\put(9.5,24){\circle*{2}}
\put(93.5,24){\circle*{2}}
\put(83.5,24){\circle*{2}}
\put(73.5,24){\circle*{2}}
\put(63.5,24){\circle*{2}}
\put(53.5,24){\circle*{2}}
\put(39.5,24){\line(-1,0){10}}
\put(29.5,24){\line(-1,0){10}}
\put(19.5,24){\line(-1,0){10}}
\put(93.5,24){\line(-1,0){10}}
\put(83.5,24){\line(-1,0){10}}
\put(73.5,24){\line(-1,0){10}}
\put(63.5,24){\line(-1,0){10}}
\bezier{120}(29.5,24)(19.5,35)(9.5,24)
\bezier{120}(39.5,24)(30.5,13)(19.5,24)
\bezier{256}(39.5,24)(23.5,52)(9.5,24)
\bezier{256}(39.5,24)(28.5,-4)(9.5,24)
\bezier{132}(83.5,24)(75.5,37)(63.5,24)
\bezier{108}(93.5,24)(82.5,15)(73.5,24)
\bezier{208}(93.5,24)(76.5,45)(63.5,24)
\bezier{176}(93.5,24)(81.5,8)(63.5,24)
\bezier{296}(93.5,24)(79.5,55)(53.5,24)
\bezier{296}(93.5,24)(84.5,-6)(53.5,24)
\put(24.5,2){\makebox(0,0)[]{$x_0$}}
\put(73.5,2){\makebox(0,0)[]{$x_1$}}
\end{picture}
\end{center}
\caption{Generators of the R. Thompson group $F$}%\label{fig:gen}
\label{f:xx}
\end{figure}
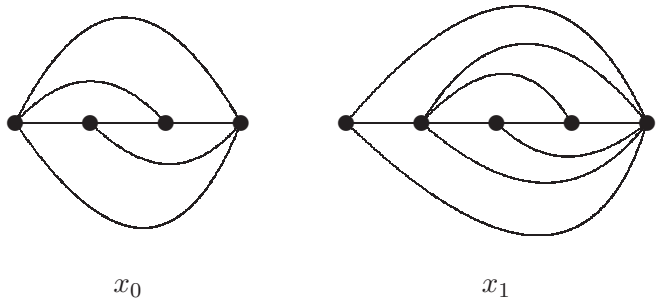

Diagrams over $\kk$ corresponding to homotopic $2$-paths are called
{\em equivalent\/}. The equivalence relation on the set of diagrams
(and the homotopy relation on the set of $2$-paths of $\kk$) can be
defined very easily as follows. We say that two cells $\pi_1$ and
$\pi_2$ in a diagram $\Delta$ over $\kk$ form a {\em dipole\/} if
$\bott(\pi_1)$ coincides with $\topp(\pi_2)$ and the labels of the
cells $\pi_1$ and $\pi_2$ are mutually inverse. Clearly, if $\pi_1$
and $\pi_2$ form a dipole, then one can remove the two cells from the
diagram and identify $\topp(\pi_1)$ with $\bott(\pi_2)$. The result
will be some diagram $\Delta'$. As in \cite{GuSa97}, it is easy to
prove that if $\delta$ is a $2$-path corresponding to $\Delta$, then
the diagram $\Delta'$ corresponds to a $2$-path $\delta'$, which is
homotopic to $\delta$. We call a diagram {\em reduced\/} if it does
not contain dipoles. A $2$-path $\delta$ in $\kk$ is called {\em reduced\/}
if the corresponding diagram is reduced.

Thus one can define morphisms in the diagram groupoid $\dd(\kk)$ as reduced
diagrams over $\kk$ with operation ``con\-cat\-enat\-ion + reduc\-tion"
(that is, the product of two reduced diagrams $\Delta$ and $\Delta'$ is the
result of removing all dipoles from $\Delta\circ\Delta'$ step by step; that process is confluent and terminating, so the result is uniquely determined \cite[Lemma 3.10]{GuSa97}).

\section{Savchuk's problem}
\label{s:sav}

\subsection{An explicit example}

Let $G$ be the subgroup of $F$ generated by the elements $y_0=x_0x_2$ and $y_1=x_1x_2$.

\begin{Lemma}\label{iso}
The subgroup $G$ is isomorphic to $F$ and the map $x_i\mapsto y_i$, $i=0,1$ extends to an isomorphism $F\to G$.
\end{Lemma}

\begin{proof}
The generators $y_0,y_1$ satisfy the defining relations of $F$ and do not commute. Since all proper homomorphic images of $F$ are Abelian \cite[Theorem 4.13]{B}, the defined homomorphism from $F$ onto $G$ is an isomorphism.
\end{proof}

\begin{Proposition}\label{eve_nor}
The group $G$ is composed of all elements $w\in F$ whose normal form is of even length. In particular, $G$ has index $2$ inside $F$.
\end{Proposition}

\begin{proof}
Clearly, the normal form of every element in $G$ is of even length. For the other direction, we will prove in several steps that for every $i,j\ge 0$, the element $x_i^{\pm 1}x_j^{\pm 1}$ belongs to $G$.

Step 1: For every $n\ge 0 $ the element $x_n^2$ belongs to  $G$.

Indeed,
$$x_0^2=(x_0x_2)(x_1x_2)^{-1}(x_0x_2)\in G.$$
Since $x_0x_2\in G$ we have  $x_2^{-1}x_0^{-1}x_0^2=x_0x_3^{-1}\in G$, hence $x_3x_0^{-1}x_0^2=x_3x_0=x_0x_4\in G$. Furthermore $x_4^{-1}x_0^{-1}x_0^2=x_0x_5^{-1}\in G$. \\
Since $x_0^2\in G$ and $x_0x_2\in G$, we have that $(x_0^2)^{x_0x_2}=x_0^2x_2x_5^{-1}\in G$, hence $x_2x_5^{-1}\in G$.
Thus, we have that $x_0x_5^{-1} x_5x_2^{-1}=x_0x_2^{-1}\in G$, so $x_2x_0^{-1}=x_0^{-1}x_1\in G$ and $x_0x_1\in G$.

Since, $x_0x_2=x_1x_0\in G$ we have $x_1x_0x_0^{-2}x_0x_1=x_1^2\in G$. Thus, for every $n\ge 0$, we have
$$(x_1^2)^{x_0^{2n}}=x_{1+2n}^2\in G.$$
Similarly, $x_0x_3^{-1},x_3^2\in G$ imply that $x_0x_3=x_2x_0\in G$. Thus $$x_2x_0x_0^{-2}x_0x_2=x_2^2\in G.$$
Conjugating by $x_0^{2n}$ shows that for any $n\ge 0$, $x_{2+2n}^2\in G$.
So finally we have that $x_n^2\in G$ for all $n\ge 0$.

Step 2: The element $x_ix_{i+1}$ belongs to $G$ for all $i\ge 0$.

Indeed, we already have that $x_0x_1,x_1x_2\in G$. Now
$$x_2x_3=(x_2x_0)(x_0^{-2})(x_0x_3)\in G.$$
Conjugating $x_1x_2$ and $x_2x_3$ by $x_0^{2n}$ for every $n$, gives that for every $n\ge 0$, the element $x_nx_{n+1}\in G.$

Step 3: The element $x_ix_j$ belongs to $G$ for every $i\le j$.

The proof is by induction on the difference $j-i$. For $j-i\in\{0,1\}$, the statement follows from Steps 1 and 2. If $j-i>1$, then
$$x_ix_j=(x_ix_{i+1})(x_{i+1}^{-1}x_j).$$
From Step 2, $x_ix_{i+1}\in G$. By the induction hypothesis and Step 1, $x_{i+1}^{-1}x_j=x_{i+1}^{-2}(x_{i+1}{x_j})\in G$.

Step 4: The element $x_i^{\pm 1}x_j^{\pm 1}$ belongs to $G$ for all $i,j\ge 0$.

Indeed, by Step 1, it's enough to consider words of the form $x_ix_j$ since $x_i^{\pm 1}x_j^{\pm 1}=(x_{i}^{-2\epsilon})(x_ix_j)(x_j^{-2\delta})$ where $\epsilon, \delta\in \{0,1\}$. Thus, Step 3 and the relation $x_ix_j=x_jx_{i+1}$ for $i>j$ complete the proof.
\end{proof}

From Proposition \ref{eve_nor}, it follows that $G=\la x_0x_2,x_0x_1,x_1x_2,x_2x_3\ra=\la\arr F, x_0x_2\ra$.
We will prove below, that for every $w\notin \arr F$, we have $\la\arr F,w\ra\supseteq G$.
We start with the following special case.

\begin{Lemma}\label{l:87} We have: $\la \arr F, x_0^2\ra=G$.
\end{Lemma}
\proof In the proof of \cite[Lemma 4.4]{GS}, we proved that $x_0^2x_1(x_0x_1x_2)^{-1}\in \arr F$. Thus,
$$ x_0x_2=x_0^{-2}(x_0^2x_1(x_0x_1x_2)^{-1})x_0^2(x_2x_3)\in\la \arr F,x_0^2\ra,$$
so $\la \arr F,x_0^2\ra$ contains $x_0x_2$ and thus, contains $G$. The opposite inclusion follows immediately from Proposition \ref{eve_nor}. \endproof

\begin{Lemma}\label{2_let}
Let $w=x_ix_j$ for $j\ge i\ge 0$ such that $j\neq i+1$. Then $\la \arr F,w\ra=G$.
\end{Lemma}

\begin{proof}
Proposition \ref{eve_nor} implies that the subgroup $\la\arr F,w\ra$ is contained in $G$. We prove the other direction in several steps.

Step 1: The lemma holds for $i=0$ and any $j\neq 1$.

Indeed, for $x_0x_2$, it follows from the definition of $G$.
The case  $j=0$ is Lemma \ref{l:87}. For $j=3$,
$$x_0^2=(x_0x_1)(x_0x_3)(x_2x_3)^{-1}\in \la \arr F,x_0x_3\ra,$$
and we are done by the previous case.
For $j>3$,
$$x_0x_{j-2}=(x_{j-3}x_{j-2})(x_0x_j)(x_{j-1}x_j)^{-1}\in \la \arr F,x_0x_j\ra,$$
so $\la \arr F, x_0x_j\ra$ contains $\la\arr F, x_0x_{j-2}\ra$ and we are done by induction.

Step 2: The lemma holds for every $i,j\ge 0$ such that $j>i+1$.

Indeed, for $i=0$ this is true by Step 1. \newline
We prove the statement for every even $i$ by induction. For $i=2$ and $j>3$,
$$x_0x_{j-2}=(x_0x_1)(x_2x_j)(x_1x_2)^{-1}\in \la \arr F, x_2x_j\ra.$$
Thus, $x_0x_{j-2}\in \la \arr F, x_2x_j\ra$ and by Step 1 we are done.
For any even $i>2$ and $j>i+1$,
$$ x_{i-2}x_{j-2}=(x_ix_j)^{(x_0x_1)^{-1}}\in\la\arr F, x_ix_j\ra.$$
Thus, by the induction hypothesis for $i-2$ we are done.

Now, suppose $i,j\ge 0$, $i$ is odd and $j>i+1$. Then
$$x_{i+1}x_{j+2}=(x_ix_j)^{-1}(x_ix_{i+1})(x_{j+1}x_{j+2})\in \la \arr F,x_ix_j\ra.$$
Since $i+1$ is even, we are done by the previous case.

Step 3: The lemma holds if $i=j$.

Indeed, for $i=j=0$, this is proved in Step 1. For any $i>0$,
$$x_{i-1}x_{i+1}=(x_{i-1}x_i)x_i^{-2}(x_ix_{i+1})\in \la \arr F,x_i^2\ra$$
and thus we are done by Step 2.
\end{proof}

In order to extend Lemma \ref{2_let} to arbitrary words $w$, we shall need new definitions and a  few auxiliary lemmas.

Let $w$ be a positive normal form. To find the normal form of the product $x_iw$ for some $i\ge 0$, one uses the rewriting rule $(\sigma_1)$ from $\Sigma$, as long as possible. That is, if $w=x_{i_1}\dots x_{i_n}$ and $i_1<i$, then $x_i$ moves over $x_{i_1}$ and becomes $x_{i+1}$. Then if for the second letter $i_2<i+1$, then $x_{i+1}$ continues moving and skips over $x_{i_2}$. %Intuitively, a \emph{block} (defined below), is a positive word (of certain form), such that if some letter $x_i$ moves over the first letter of $w$, then it moves over the entire $w$. The easiest example of a block is a word $x_ix_{i+1}$ for some $i\ge 0$.

\begin{Definition}
Let $w=x_{i_1}\dots x_{i_n}$ be a positive normal form. We say that the letter $x_i$, $i\ge 0$ \emph{skips over} $w$ if $x_iw=wx_{i+n}$ in $F$.
\end{Definition}

\begin{Lemma}\label{l:skips}
Let $w=x_{i_1}\dots x_{i_n}$ be a positive normal form and $x_i$, $i\ge 0$ a letter. The following assertions hold.
\begin{enumerate}
\item[(1)] The letter $x_i$ skips over $w$ if and only if for all $j=1,\dots,n$ we have $i_j<i+j-1$.
\item[(2)] If $x_i$ skips over $w$, then for all $k>i$ the letter $x_k$ also skips over $w$.
\end{enumerate}
\end{Lemma}

\begin{proof}
Part (2) follows immediately from Part (1). Indeed, if the inequalities in (1) hold for $i$ then they hold for all $k>i$.
Let us prove Part (1). If for all $j=1,\dots,n$ we have $i_j<i+j-1$, then $x_i$ clearly skips over $w$. For the other direction, assume that $x_iw=wx_{i+n}$ in $F$ but there exists $j=1,\dots,n$ such that $i_j\geq i+j-1$. Let $r$ be the minimal such $j$ and $w_2$ the suffix of $w$ starting from $x_{i_r}$. Let $w_1$ be the prefix of $w$ of length $r-1$, so that $w=w_1w_2$. Applying the rewriting rule $(\sigma_1)$ to $x_iw$ as long as possible, we get the normal form $w_1x_{i+r-1}w_2$.

We claim that there must exist a sequence of applications of rules from $\Sigma$ (a {\em derivation}) turning $wx_{i+n}$ into $w_1x_{i+r-1}w_2$. Indeed, since $\Sigma$ is confluent and terminating there is a unique semi-normal form which one can get by applying a sequence of applications of rules from $\Sigma$ to the word $wx_{i+n}$. The resulting semi-normal form is necessarily positive, and as such, it is a normal form. By uniqueness, it must be equal to $w_1x_{i+r-1}w_2$.

Since the word $wx_{i+n}$ is positive, a derivation applied to it must consist entirely of applications of $(\sigma_1)$. Each application of $(\sigma_1)$ increases the sum of indices of letters in the word. In contradiction to the sum of indices of letters in $wx_{i+n}=w_1w_2x_{i+n}$ being strictly greater than the corresponding sum for the normal form $w_1x_{i+r-1}w_2$.
\end{proof}

\begin{Definition}
Let $w=x_{i_1}\dots x_{i_n}$ be a positive normal form. We say that $w$ is a \emph{block} if $x_{i_1}\neq x_{i_n}$ and for every  $j=1,\dots,n$ we have $i_j<i_1+j$.
\end{Definition}

\begin{Remark}\label{r:block}
Let $w=x_{i_1}\dots x_{i_n}$ be a positive normal form. It follows form Lemma \ref{l:skips} that $w$ is a block if and only if
\begin{enumerate}
\item[(1)] $w$ contains at least two distinct letters; and
\item[(2)] $x_{i_1+1}$ skips over $w$; equivalently, $x_k$ skips over $w$ for all $k>i_1$.
\end{enumerate} %$w$ contains at least two distinct letters and $x_{i_1+1}$ skips over $w$. Since the equalities $x_iw=wx_{i+n}$ and $x_i^{-1}w=wx_{i+n}^{-1}$ are equivalent, we get the following characterization of blocks.
Since the equalities $x_iw=wx_{i+n}$ and $x_i\iv w =wx_{i+n}\iv$ in $F$ are equivalent, one could replace the second condition of Remark \ref{r:block} by the condition: $x_k^{\pm 1}w=wx_{k+n}^{\pm 1}$ for all $k> i_1$.
\end{Remark}

%The word block is used to convey the idea that for every $i\ge 0$, of $x_i$ skips over the first letter of $w$, then it skips over the entire $w$. Examples of blocks: $x_0x_1$, $x_2^2x_{4}^3x_7$, etc.

Intuitively, a block is a positive normal form with at least two distinct letters, such that for all $i$, if the letter $x_i$ skips over the first letter of $w$, then it skips over the entire $w$. Examples of blocks: $x_0x_1$, $x_2^2x_{4}^3x_7$, etc.

\begin{Lemma}\label{3_pro}
Let $B=x_{i_1}\dots x_{i_n}$ be a block.
Then
\begin{enumerate}
\item [(1)] If $B'$ is a \emph{translation} of $B$, i.e., $B'=x_{i_1+k}\dots x_{i_n+k}$ for some $k\ge 0$, then $B'$ is also a block.
\item[(2)] For every $j\neq i_1$ we have $x_j^{-1}B=B'x_r^{-1}$, where $B'$ is a translation of $B$ and $r\ge 0$.
\item [(3)] If $B_1=x_{i_2}\dots x_{i_n}$ is not a block (in which case, we say that $B$ is a \emph{minimal} block), then there exists $j\in\{2,\dots,n\}$ such that $i_j=i_1+j-1$.
\end{enumerate}
\end{Lemma}

\begin{proof}
(1) Follows easily from the definition.

(2) If $j>i_1$, then by Remark \ref{r:block} $x_j^{\pm 1}B=Bx_{j+n}^{\pm 1}$. Otherwise, the rewriting rule $x_j^{-1}x_i\rightarrow x_{i+1}x_j^{-1}$ for $i>j$, shows that $x_j\iv B=B'x_j^{-1}$, where $B'$ is a translation of $B$ (with the translation number $k=1$).

(3) If $x_{i_2}=x_{i_1+1}$, then $B$ satisfies $(3)$ for $j=2$. Otherwise, $B$ being a block, implies that $i_2=i_1$ and thus the length $|B|=n>2$. It follows that the first and last letters of $B_1$ are distinct. Thus, if $B_1$ is not a block, then there exists some $j=2,\dots,n$  such that $i_j\ge i_2+j-1=i_1+j-1$. Now the fact that $B$ is a block guarantees that $i_j<i_1+j$. Thus, $i_j=i_1+j-1$.
\end{proof}

\begin{Lemma}\label{pos}
Every double coset $C=\arr F a \arr F$ of $\arr F$ contains a positive element $w$ such that for every $w'\in C$, we have $|w|\le |w'|$.
\end{Lemma}

\begin{proof}
Let $w\in C$ be an element of minimal length. If $w$ is positive, we are done. Otherwise, it is possible to replace $w$ by a positive word, without increasing its length, as follows. Let $x_i^{-1}$ be the last letter in the normal form of $w$.  Replacing $w$ by $w(x_ix_{i+1})$ yields a word in $C$ of the same length, with the negative part of the normal form being shorter. Thus, proceeding by induction, gives the result.
\end{proof}

\begin{Lemma}\label{no_blo}
Let $C$ be a double coset of $\arr F$ and let $w$ be a positive normal form of minimal length in $C$. Then $w$ does not contain any block.
\end{Lemma}

\begin{proof}
We begin by making a general observation.
Let $W$ be a normal form which contains a block $B$. Thus, we have $W=w_1Bw_2$ for some possibly empty $w_1$ and $w_2$. It follows easily from Lemma \ref{3_pro}(2) that if one multiplies $W$ from the left by $x_j^{-1}$ for some $j$, then
either $x_j^{-1}W$ is shorter than $W$, or $x_j^{-1}W$ is of the form $w_1'B'w_2'$ for some $w_1',B',w_2'$ such that $|w_1'|=|w_1|$, $|w_2'|=|w_2|+1$ and such that $B'$ is a translation of $B$.

Now, let $C=\arr F w\arr F$ be a double coset of $\arr F$ and use Lemma \ref{pos} to assume that $w$ is a positive normal form of minimal length in $C$. We claim that $w$ does not contain any block $B$.
Indeed, assume by contradiction that $w$ is of the form $w=w_1Bw_2$ for some minimal block $B$.
It is possible to replace $w$ by an element $w'\in C$ of the same length, such that $w'=B'w_2'$ where $B'$ is a translation of $B$, as follows.
If $w_1$ is empty, we take $w'=w$. Otherwise, let $w_1=x_jw_1''$.
Then $w'':=(x_jx_{j+1})^{-1}w=x_{j+1}^{-1}w_1''Bw_2 \in C$. By the above observation and the minimality of $w$ we have that $w''=w_1'B'w_2'$, when $|w_1'|=|w_1''|=|w_1|-1$, $|w_2'|=|w_2|+1$ and $B'$ is a translation of $B$. Proceeding by induction on the length of the prefix $w_1$, we get the result.

Thus, we shall assume that $w=Bw_2$. Let $B=x_{i_1}\dots x_{i_n}$. Then
$$t=(x_{i_1}x_{i_1+1})^{-1}w=x_{i_1+1}^{-1}x_{i_1}^{-1}Bw_2=x_{i_1+1}^{-1}x_{i_2}\dots x_{i_n}w_2\in C.$$
By lemma \ref{3_pro}(3), the minimality of the block $B$ implies that there exists some $j=2,\dots n$ such that $i_j=i_1+j-1$. If $j$ is the minimal such index, then the letter $x_{i_j}$ cancels in $t$ with $x_{i_1+1}^{-1}$, in contradiction with  $w$ having minimal length in $C$ (by Lemma \ref{pos}).
\end{proof}

\begin{Theorem}\label{main}
Let $w\in F\setminus \arr F$. Then the group $\la \arr F, w\ra$ contains $G$.
If the normal form of $w$ is of even length, then $\la \arr F,w\ra=G$.
Otherwise, $\la \arr F,w\ra=F$.
\end{Theorem}

\begin{proof}
We denote $H=\la \arr F,w\ra$. It suffices to prove that $H\supseteq G$, since the rest of the theorem follows easily from Proposition \ref{eve_nor}. Clearly, we can replace $w$ by any element of the double coset $C=\arr Fw\arr F$. Thus, by Lemmas \ref{pos} and \ref{no_blo}, we can assume that $w$ is positive and does not contain any block.
If $w=x_i$ for some $i\ge 0$, then $H=F$. If $w$ is of length $2$, then $w=x_ix_j$ such that $j\neq i+1$ (indeed, $w$ is not a block). Thus, we are done by Lemma \ref{2_let}.
If $|w|>2$, we write $w=w'x_j^n$ where $j\ge 0$, $n\ge 1$ and $w'$ is either empty, or its last letter is not $x_j$.
If $w'$ is empty, then
%$$x_0x_{1+n}=(x_0x_1)^{x_0^n}\in H,$$ and we are done by Lemma \ref{}.
$$x_jx_{j+1+n}=(x_jx_{j+1})^{x_j^n}\in H$$
and we are done by Lemma \ref{2_let}.
Assume that $w'$ is not empty and let $k=|w'|$.
If  $$(*)\ \ j\ge k\ \hbox{ and }\ x_{j-k}w'=w'x_j,$$ then
$x_{j-k}$ skips over $w'$ and so $x_{j-k+1}$ also skips over $w'$ (Lemma \ref{l:skips}(2)). As such we have that
%$$(x_{j-k}x_{j-k+1})^w=(w'x_j^n)^{-1}(x_{j-k}x_{j-k+1})w'x_j^n
%=(w'x_j^n)^{-1}w'(x_{j}x_{j+1})x_j^n =(w'x_j^n)^{-1}w'x_j^{n+1}x_{j+n+1}=x_jx_{j+n+1}\in H$$
\begin{equation*}
\begin{split}
(x_{j-k}x_{j-k+1})^w & =(w'x_j^n)^{-1}(x_{j-k}x_{j-k+1})w'x_j^n=(w'x_j^n)^{-1}w'(x_{j}x_{j+1})x_j^n \\
& = (w'x_j^n)^{-1}w'x_j^{n+1}x_{j+n+1}=x_jx_{j+n+1}\in H
\end{split}
\end{equation*}
and we are done by Lemma \ref{2_let}.

Thus it suffices to prove that $(*)$ must hold. Assume by contradiction that it doesn't.
Note, that in that case, for every $i\ge 0$, if $x_i$ skips over $w'$ then $x_i$ skips over the entire $w$. That is, if for some $i$, we have $x_iw'=w'x_{i+k}$, then $x_iw'x_j^n= w'x_j^nx_{i+k+n}$.
Indeed, $x_iw'=w'x_{i+k}$ implies that $x_iw'x_j^n=(w'x_{i+k})x_j^n$. If $i+k>j$, the result is clear. If $d=j-(i+k)\ge 0$, then since $x_i$ skips over $w'$ and $i+d\ge i$, by Lemma \ref{l:skips}(2) we have $x_{i+d}w'=w'x_j$, in contradiction to $(*)$ not holding (note that $i+d=j-k\ge 0$).

Let $i$ be the minimal index such that $x_iw'=w'x_{i+k}$. %(i.e., the minimal $i$ such that $x_i$ is ``large enough" to jump over the entire $w'$).
Let $w'=x_{i_1}\dots x_{i_k}$. We claim that there is some index $i_r$ such that $i_r=i+r-2$.
Indeed, $x_i$ skipping over $w'$, implies that for all $r=1,\dots k$, we have $i_r<i+r-1$ (Lemma \ref{l:skips}(1)).
If for all $r$, we have that $i_r<i+r-2$, then one can replace $x_i$ with $x_{i-1}$, by contradiction to the minimality of $i$ (note that $i\ge 1$).

Let $B$ be the suffix of $w$ starting from $x_r$ for $r$ as above. That is, $B= x_{i_r}\dots x_{i_k}x_j^n$ Let $A$ be the prefix of $w$ of length $r-1$ so that $w=AB$.
Then from the above, $x_iw=x_iAB=Ax_{i+r-1}B=ABx_{i+k+n}$. In particular, $x_{i+r-1}$ skips over $B$.
Since the first letter of $B$ is $x_{i_r}=x_{i+r-2}$ and the first and last letters of $B$ are distinct, $B$ is a block inside $w$ (by Remark \ref{r:block}), in contradiction to the choice of $w$.

%\begin{Lemma}\label{con}
%In the above notations, if (*) doesn't hold, then $w$ contains a block.
%\end{Lemma}
%
%\begin{proof}
%%Assume by contradiction that it doesn't.
%First note that if $(*)$ doesn't hold, then for every $i\ge 0$, if $x_i$ ``moves over $w'$" then it ``moves over the entire $w$". That is, if for some $i$, we have $x_iw'=w'x_{i+k}$, then $x_iw'x_j^n= w'x_j^nx_{i+k+n}$.
%Indeed, $x_iw'=w'x_{i+k}$ implies that $x_iw'x_j^n=(w'x_{i+k})x_j^n$. If, $i+k>j$, the result is clear. If $d=j-(i+k)\ge 0$, then $x_{i+d}w'=w'x_j$, in contradiction to $(*)$ not holding (note that $i+d$  must be equal to $j-k$).
%
%Let $i$ be the minimal index such that $x_iw'=w'x_{i+k}$. %(i.e., the minimal $i$ such that $x_i$ is ``large enough" to jump over the entire $w'$).
%Let $w'=x_{i_1}\dots x_{i_k}$. We claim that there is some index $i_r$ such that $i_r=i+r-2$.
%Indeed, $x_i$ jumping over $w'$, implies that for all $r=1,\dots k$, we have $i_r<i+r-1$.
%If for all $j$, we have that $i_r<i+r-2$, then one can replace $x_i$ with $x_{i-1}$, by contradiction to the minimality of $i$.
%
%Let $B$ be the suffix of $w$ starting from $x_r$ for $r$ as above. That is, $B= x_{i_r}\dots x_{i_n}x_j^n$ Let $A$ be the prefix of $w$ such that $w=AB$.
%Then, from the above, $x_iw=x_iAB=Ax_{i+r-1}B=ABx_{i+k+n}$. In particular, $x_{i+r-1}$ moves over $B$.
%Since the first letter of $B$ is $x_{i_r}=x_{i+r-2}$ and the first and last letters of $B$ are distinct, $B$ is the required block.
%\end{proof}
\end{proof}

\begin{Remark}
There are several descriptions of the subgroup $G$ of $F$. Vaughan Jones in \cite{Jones} mentioned that $\arr F$ is contained in the subgroup $G'$ of  $F$ consisting of functions whose slope at 1 is of the form $2^{2n}$, $n\in \Z$. In \cite{Bleak} that subgroup is denoted by $K_{(1,2)}$ (it is proved there that this subgroup is isomorphic to $F$). Clearly Theorem \ref{main} shows that $G'=G$.
\end{Remark}

Saul Schleimer asked (private communication) whether there are subgroups strictly between $\arr F$ and $G$. Theorem \ref{main} answers that question:

\begin{Corollary}\label{cor}
$\arr F$ is a maximal subgroup of infinite index inside $G$.
\end{Corollary}

The next theorem gives a counterexample to Savchuk's problem \cite[Problem 1.5]{Sav}.
\begin{Theorem}\label{t:11}
Let $\Psi\colon F \to G$ be the isomorphism taking $x_i$ to $y_i$, $i=0,1$ (see Lemma \ref{iso}). Then $H=\Psi^{-1}(\arr F)$ is a maximal subgroup of infinite index in $F$ and $H$ does not stabilize any $x\in (0,1)$.
\end{Theorem}

\begin{proof}
It follows from Corollary \ref{cor} that $H$ is a maximal subgroup of $F$ of infinite index.
Recall that $y_0=x_0x_2$ and $y_1=x_1x_2$. One can check that
$$x_0x_1=y_0y_1y_0^{-1}y_1^{-1}y_0.$$
%$$x_0x_1=y_1^{-1}y_0y_1y_0^{-2}y_1y_0y_1^{-1}y_0.$$
Thus,
$$\Psi^{-1}(x_0x_1)=x_0x_1x_0^{-1}x_1^{-1}x_0=x_0x_1x_2^{-1}\in H.$$
%$$\Psi(x_0x_1)=x_1^{-1}x_0x_1x_0^{-2}x_1x_0x_1^{-1}x_0=x_0x_1x_2^{-1}\in H.$$
On binary fractions, $g=x_0x_1x_2^{-1}$ acts as follows:
\[g(t)=
  \begin{cases}
   .0\alpha & \hbox{ if }   t=.00\alpha  \\
   .10\alpha  & \hbox{ if } t=.010\alpha  \\
   .1100\alpha & \hbox{ if } t=.011\alpha  \\
   .1101\alpha & \hbox{ if } t=.10\alpha  \\
   .111\alpha & \hbox{ if }  t=.11\alpha
  \end{cases}
\]
It is clear from the binary description, that $x_0x_1x_2^{-1}$ does not stabilize any $x\in (0,1)$. Indeed, $\Psi$ changes a prefix of every binary fraction other than $0=.0^N$ and $1=.1^N$.
\end{proof}

Since $\arr F$ is isomorphic to the Thompson-Brown group $F_3$ \cite{GS}, we have

\begin{Corollary} The R. Thompson group $F$ has a maximal subgroup which is isomorphic to $F_3$.
\end{Corollary}

\subsection{Implicit examples} \label{ss:imp}

Here is another idea how to prove existence of maximal subgroups of infinite index in $F$ which do not fix a point in $(0,1)$. We will show below that this idea actually works.

Let $g,h$ be two elements of $F$ such that $M_0=\la g,h\ra$ does not fix a point in $(0,1)$. Suppose that the images of $g,h$ in the free Abelian group $F/[F,F]$ generate the whole group $F/[F,F]$. Then $M_0$ cannot be contained in any non-trivial subgroup of finite index in $F$ because subgroups of finite index all contain the derived subgroup $[F,F]$. If $M_0\ne F$, then $M_0$ is contained in a maximal subgroup $M$ of $F$ (by Zorn's lemma since $F$ is finitely generated), and $M$ is of infinite index in $F$ and does not fix any point in $(0,1)$. The problem is how to prove that $M_0\ne F$. As we have mentioned in the introduction there is no known algorithm to decide when two (or any finite number $> 1$) of elements of $F$ do not generate the whole $F$. But here is a partial algorithm which works quite often even when the images of these elements in $F/[F,F]$ generate the whole $F/[F,F]$.

Recall the procedure (first discovered by Stallings \cite{Sta}) of checking if an element $g$ of a free group $F_n$ belongs to the subgroup $H$ generated by elements $h_1,...,h_k$. Take paths labeled by $h_1,...,h_k$. Identify the initial and terminal vertices of these paths to obtain a bouquet of circles $K'$ with a distinguished vertex $o$. Then \emph{fold} edges as follows: if there exists a vertex with two outgoing edges of the same label, we identify the edges. As a result of all the foldings (and removing the hanging trees), we obtain the \emph{Stallings core} of the subgroup $H$ which is a finite automaton $A(H)$ with $o$ as its input/output vertex. Then $g\in H$ if and only if $A(H)$ accepts the reduced form of $g$. It is well known that the Stallings core does not depend on the generating set of the subgroup $H$.

In the case of diagram groups we will use the same strategy but instead of automata we will have directed complexes and instead of words - diagrams. We give the construction for general diagram groups and then apply it to Thompson group $F$. %We consider general diagram groups and then apply the construction to Thompson group $F$.

\begin{Definition} Let $\kk=\la E_{\kk}\mid \F_{\kk}\ra$ be a directed 2-complex. A \emph{2-automaton} over $\kk$ is a directed 2-complex $\La=\la E_{\La}\ |\F_{\La}\ra$ with two distinguished 1-paths $p_{\La}$ and $q_{\La}$ (the input and output 1-paths), together with a map $\phi$ from $\La$ to $\kk$ which takes vertices to vertices, edges to edges and cells to cells, which is a homomorphism of directed graphs and commutes with the maps $\topp, \bott$ and $\iv$.
%$\phi_{\La}\colon \F_{\La}\to\F_{\kk}$ which commutes with the maps $\topp, \bott$ and $\iv$ and with two distinguished 1-paths $p_{\La}$ and $q_{\La}$ (the input and output 1-paths).
 We shall call $\phi$ an \emph{immersion}.
\end{Definition}

For example, every diagram $\Delta$ over $\kk$ is a 2-automaton with a natural immersion $\phi_\Delta$ and the distinguished paths $\topp(\Delta)$ and $\bott(\Delta)$.

\begin{Definition} Let $\La, \La'$ be two 2-automata over $\kk$. A map $\psi$ from $\La'$ to $\La$ which takes vertices to vertices, edges to edges and cells to cells, which is a homomorphism of directed graphs and commutes with the maps $\topp, \bott, \iv$ and the immersions is called a \emph{morphism} from $\La'$ to $\La$ provided $\psi(p_{\La'})=p_{\La}, \psi(q_{\La'})=q_{\La}$.
\end{Definition}

\begin{Definition} We say that a 2-automaton $\La$ over $\kk$ \emph{accepts} a diagram $\Delta$ over $\kk$ if there is a morphism $\psi$ from the 2-automaton $\Delta$ to the 2-automaton $\La$.
\end{Definition}

Let $\Delta$ be a diagram over $\kk$ with top and bottom paths having the same labels (i.e., a \emph{spherical diagram} in terminology of \cite{GuSa97}). Let us identify $\topp(\Delta)$ with $\bott(\Delta)$. We can view the result as a  2-automaton $\La'$ over $\kk$ drawn on a sphere with
the distinguished paths $p=q=\topp(\Delta)=\bott(\Delta)$. The 2-automaton $\La'$ clearly accepts $\Delta$, and any diagram of the form $\Delta\circ \Delta\circ ...$.

Suppose we want a 2-automaton  that accepts all reduced diagrams that are equal to $\Delta^n$ for some $n$. Then we should do an analog of the Stallings foldings. Namely, let $\La'$ be the 2-automaton as above.  Now every time we see two cells that have the same image under the immersion of $\La'$ and share the top (resp. bottom) 1-path, then we identify their bottom (resp. top) 1-paths and identify the cells too. This operation is called \emph{folding of cells} (see \cite[Remark 8.8]{GSdc}).
Clearly the result is a directed 2-complex and the immersion of $\La'$ induces an immersion of the new directed 2-complex. Thus we again get a 2-automaton. Let $\La$ be the 2-automaton obtained after all possible foldings in $\La'$.   Then $\Delta\circ \Delta\circ...$ is still accepted by $\La$. But also the reduced diagrams $\Delta^n$ are accepted by $\La$ for every $n$.

More generally, let $\Delta_1, ...,\Delta_n$ be (reduced) diagrams from the diagram group $\DG(\kk,u)$ i.e., diagrams over $\kk$ with the same label $u$ of their top and bottom paths. Then we can identify all $\topp(\Delta_i)$ with all $\bott(\Delta_i)$ and obtain a 2-automaton $\La'$ over $\kk$ with the distinguished paths $p=q=\topp(\Delta_i)=\bott(\Delta_i)$. We can view $\La'$ as a ``bouquet of spheres''. That automaton accepts any concatenation of diagrams $\Delta_i$ and their inverses. Let us do all the foldings in $\La'$. We obtain a 2-automaton $\La$.

The proof of the following lemma is quite similar to the proof of \cite[Lemma 3.10]{GuSa97} and is left to the reader as an exercise.

\begin{Lemma} \label{l:88} The rewriting system where objects are finite 2-automata over a directed 2-complex $\kk$ and moves are foldings, is confluent and terminating.
\end{Lemma}

By Lemma \ref{l:88} and \cite[Lemma 1.7.8]{Sbook} the 2-automaton $\La$ is uniquely determined by $\La'$.

\begin{Remark} Since folding is a ``local'' operation, Lemma \ref{l:88} implies that even if we start with an infinite directed 2-automaton over a directed 2-complex $\kk$, the end result of all the foldings does not depend on the order in which the foldings are performed.
\end{Remark}

\begin{Lemma} \label{l:au} The 2-automaton $\La$ accepts all reduced diagrams from the subgroup of the diagram group $\DG(\kk,u)$ generated by $\Delta_1,\ldots , \Delta_n$.
\end{Lemma}

\proof Suppose that cells $\pi_1$ and $\pi_2\iv$ form a dipole in a diagram $\Delta$ accepted by $\La$, $p_1=\topp(\pi_1)$, $q_1=\bott(\pi_1)=\bott(\pi_2)$ , $p_2=\topp(\pi_2)$, and
$\Delta'$ is obtained from $\Delta$ by removing this dipole. We claim that then $\Delta'$ is also accepted by $\La$. This will immediately imply the statement of the lemma. To prove the claim let $\phi$ be a morphism from $\Delta$ to $\La$ which takes both $\topp(\Delta)$ and $\bott(\Delta)$ to the input/output 1-path of $\La$. Then the images of the cells $\pi_1$, $\pi_2$ are cells in $\La$ that have the same images under the immersion of $\La$ into $\kk$  and share a common top or  bottom path. Hence these two cells must be the same in $\La$. Therefore $\phi(p_1)=\phi(p_2)$. Hence $\phi$ induces a morphism $\bar\phi$ from $\Delta'$ into $\La$. The morphism $\bar\phi$ coincides with $\phi$ on $\topp(\Delta)$ and $\bott(\Delta)$. Hence $\La$ accepts $\Delta'$.\endproof

\begin{Remark} Suppose $\Delta_1',...,\Delta_m'$ is another generating set of $H=\la \Delta_1,...,\Delta_n\ra$ and all diagrams $\Delta_i'$ are reduced. Let $\La'$ be the 2-automaton corresponding to $\Delta_1',...,\Delta_m'$. It is easy to see that since $\Delta_i$ is accepted by $\La'$ and $\Delta_j'$ is accepted by $\La$ for every $i,j$, we have $\La'=\La$. Thus $\La$ does not depend on the choice of generating set of the subgroup $H$ (we can even use any infinite generating set, the resulting 2-automaton will be the same). Thus $\La$ can be called the {\em Stallings 2-core} of the subgroup $H$.
\end{Remark}

\begin{Remark} In general the 2-automaton $\La$ constructed above can also accept diagrams not from the subgroup $H=\la \Delta_1,\ldots, \Delta_n\ra$ but in many cases it accepts only the diagrams from $H$. In that case we call $H$ {\em closed} (see Section \ref{5}).
%
%In that case the diagram group $\dd(\La,u)$ is isomorphic to $H$. This is a powerful method of representing subgroups of diagram groups as diagram groups themselves. Although this method was not published before it was discussed by Victor Guba and the second author of this paper more than 15 years ago. It was used implicitly, for example, in \cite{GuSa99} to find such a presentation of the derived subgroup $[F,F]$ of the R. Thompson group $F$ (see  \cite[Theorem 26]{GuSa99}).
\end{Remark}

Lemma \ref{l:au} can be used to construct many proper subgroups of $F$ that 1) do  not fix a point in $(0,1)$ and 2) whose image in $F/[F,F]$ is the
whole $F/[F,F]$ (and hence prove existence of examples solving Savchuk's problem). We illustrate this by the following

\begin{Lemma} The subgroup $H=\la x_0, x_1x_2x_1\iv\ra$ is a proper subgroup of $F$ and satisfies the two properties 1) and 2) above.
\end{Lemma}

\proof  It is easy to check that $x_0$ does not fix a point in $(0,1)$ which implies 1). The image of $x_1x_2x_1\iv$ in $F/[F,F]$ coincides with the image of $x_1$. Hence $H[F,F]=F$ which implies 2). It remains to prove that $H$ is a proper subgroup of $F$. We shall prove that $x_1\not\in H$.

Let us denote the positive cell of the Dunce hat $\la x\mid x\to x^2\ra$ by $\pi$. The diagrams  for $x_0$ and $x_1x_2x_1\iv$ viewed as 2-automata are in Figure \ref{f:xxxx} below (the immersion to the Dunce hat maps all positive cells to $\pi$, and all edges to the only edge of the Dunce hat).

\begin{figure}[ht]
\begin{center}
% This is a LaTeX picture output by TeXCAD.
% File name: [p51.pic].
% Version of TeXCAD: 4.3
% Reference / build: 30-Jun-2012 (rev. 105)
% For new versions, check: http://texcad.sf.net/
% Options on the following lines.
%\grade{\on}
%\emlines{\off}
%\epic{\off}
%\beziermacro{\on}
%\reduce{\on}
%\snapping{\off}
%\pvinsert{% Your \input, \def, etc. here}
%\quality{8.000}
%\graddiff{0.005}
%\snapasp{1}
%\zoom{4.0000}
\unitlength .7mm % = 2.845pt
\linethickness{0.4pt}
\ifx\plotpoint\undefined\newsavebox{\plotpoint}\fi % GNUPLOT compatibility
\begin{picture}(181.291,100)(0,0)
\put(21.75,43.25){\line(1,0){63.5}}
\qbezier(21.75,43.25)(56.25,101.375)(84.75,43)
\qbezier(22,43.5)(48,71.625)(66,43.25)
\qbezier(22,43.25)(57,-15.75)(85,43.25)
\qbezier(44.5,43.25)(62.875,6.25)(84.75,43.25)
\put(22,43.25){\circle*{1.581}}
\put(95.5,42.75){\circle*{1.581}}
\put(105.75,43){\circle*{1.581}}
\put(117,43){\circle*{1.581}}
\put(139.75,42.75){\circle*{1.581}}
\put(163.5,42.5){\circle*{1.581}}
\put(180.5,42.75){\circle*{1.581}}
\put(45,43.25){\circle*{1.581}}
\put(65.75,43.5){\circle*{1.581}}
\put(84.75,43.75){\circle*{1.581}}
\put(59.75,60.5){\makebox(0,0)[cc]{$\gamma_1$}}
\put(44.75,49.5){\makebox(0,0)[cc]{$\gamma_2$}}
\put(64,34.75){\makebox(0,0)[cc]{$\gamma_3$}}
\put(40.75,29.5){\makebox(0,0)[cc]{$\gamma_4$}}
\put(54.75,75){\makebox(0,0)[cc]{$e_1$}}
\put(44,60.5){\makebox(0,0)[cc]{$e_2$}}
\put(32.5,46.25){\makebox(0,0)[cc]{$e_3$}}
\put(55,45.75){\makebox(0,0)[cc]{$e_4$}}
\put(75,46){\makebox(0,0)[cc]{$e_5$}}
\put(61,23){\makebox(0,0)[cc]{$e_6$}}
\put(53.75,11.25){\makebox(0,0)[cc]{$e_7$}}
\put(94.5,43){\line(1,0){86.5}}
\qbezier(94.75,43)(132.375,140)(180.5,43)
\qbezier(180.5,43)(146.5,-41.625)(94.5,43.25)
\qbezier(105,42.75)(132.25,113.25)(180.5,42.75)
\qbezier(105.25,43)(136.5,86.5)(163.75,43)
\qbezier(116.75,43)(134.75,69.375)(163.75,42.25)
\qbezier(105.5,43)(145.875,-17.5)(179.75,43)
\qbezier(105.75,43)(130.875,12)(139.5,43)
\qbezier(139.5,43)(158.375,13.5)(179.75,43)
\put(132,83.75){\makebox(0,0)[cc]{$\gamma_5$}}
\put(135.5,70.75){\makebox(0,0)[cc]{$\gamma_6$}}
\put(130,58.5){\makebox(0,0)[cc]{$\gamma_7$}}
\put(135.75,48.25){\makebox(0,0)[cc]{$\gamma_8$}}
\put(159.75,36.25){\makebox(0,0)[cc]{$\gamma_9$}}
\put(124.75,36.5){\makebox(0,0)[cc]{$\gamma_{10}$}}
\put(143.25,23.5){\makebox(0,0)[cc]{$\gamma_{11}$}}
\put(136.5,6.5){\makebox(0,0)[cc]{$\gamma_{12}$}}
\put(138,94.5){\makebox(0,0)[cc]{$e_8$}}
\put(117.75,72.5){\makebox(0,0)[cc]{$e_9$}}
\put(99,45.5){\makebox(0,0)[cc]{$e_{10}$}}
\put(151.5,61.75){\makebox(0,0)[cc]{$e_{11}$}}
\put(111.75,46){\makebox(0,0)[cc]{$e_{12}$}}
\put(138.5,58.25){\makebox(0,0)[cc]{$e_{13}$}}
\put(127,45.5){\makebox(0,0)[cc]{$e_{14}$}}
\put(148.25,46){\makebox(0,0)[cc]{$e_{15}$}}
\put(169.75,46){\makebox(0,0)[cc]{$e_{16}$}}
\put(124.75,31.25){\makebox(0,0)[cc]{$e_{17}$}}
\put(158.25,30.75){\makebox(0,0)[cc]{$e_{18}$}}
\put(143.75,17){\makebox(0,0)[cc]{$e_{19}$}}
\put(162.5,6.25){\makebox(0,0)[cc]{$e_{20}$}}
\end{picture}
\end{center}
\caption{The diagrams for $x_0$ and $x_1x_2x_1\iv$}
\label{f:xxxx}
\end{figure}
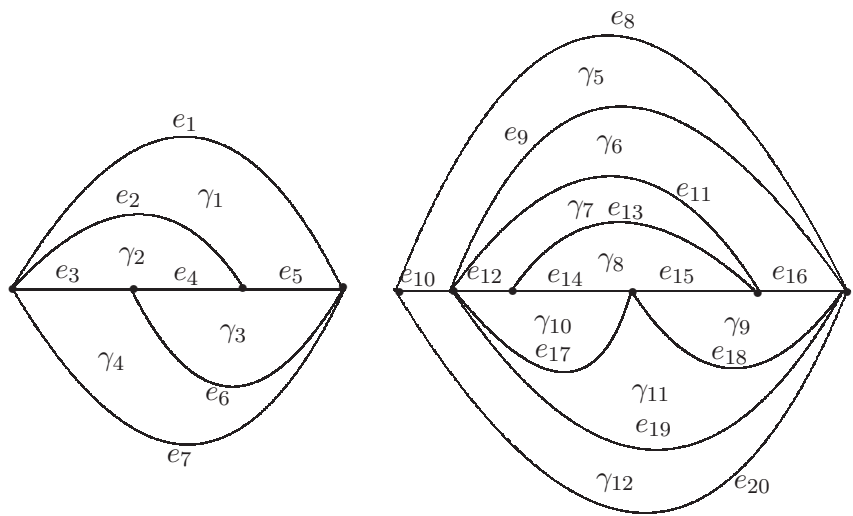

Together the two diagrams have 20 edges (labeled by $e_1,\ldots, e_{20}$) and 12 cells (labeled $\gamma_1, \ldots,  \gamma_{12}$). To construct the 2-automaton $\La$ for these two diagrams, we first need to identify the top and the bottom paths of both diagrams. So we set $e_1=e_7=e_8=e_{20}$. Now the positive cells $\gamma_1, \gamma_4, \gamma_5, \gamma_{12}$ need to be folded because these cells share the top 1-path $e_1$. So we need to identify $\gamma_1=\gamma_4=\gamma_5=\gamma_{12}$ and the edges
$e_2=e_3=e_{10}$ and $e_5=e_6=e_9=e_{19}$. Now the cells $\gamma_3$, $\gamma_6$, $\gamma_{11}$ have common top edge $e_5$. So we need to fold these three cells. Thus  $\gamma_3=\gamma_6=\gamma_{11}$, $e_4=e_{11}=e_{17}$, $e_5=e_{16}=e_{18}$.  Then the cells $\gamma_7$ and $\gamma_{10}$ share the top edge $e_4$. So we set $\gamma_7=\gamma_{10}$, $e_{13}=e_{14}$. Furthermore $\gamma_9$ and $\gamma_3$ now share the top edge $e_5$. So we need to set $e_4=e_{15}$. No more foldings are needed, and the $2$-automaton $\La$ is presented in Figure \ref{f:x1} (there the cells and edges are supposed to be identified according to their labels: all $e_1$ edges are the same, all $\gamma_7$-cells are the same, etc.).

\begin{figure}[ht!]
\begin{center}
% This is a LaTeX picture output by TeXCAD.
% File name: [p52.pic].
% Version of TeXCAD: 4.3
% Reference / build: 30-Jun-2012 (rev. 105)
% For new versions, check: http://texcad.sf.net/
% Options on the following lines.
%\grade{\on}
%\emlines{\off}
%\epic{\off}
%\beziermacro{\on}
%\reduce{\on}
%\snapping{\off}
%\pvinsert{% Your \input, \def, etc. here}
%\quality{8.000}
%\graddiff{0.005}
%\snapasp{1}
%\zoom{4.0000}
\unitlength .6 mm % = 2.845pt
\linethickness{0.4pt}
\ifx\plotpoint\undefined\newsavebox{\plotpoint}\fi % GNUPLOT compatibility
\begin{picture}(181.291, 90)(0,0)
\put(21.75,43.25){\line(1,0){63.5}}
\qbezier(21.75,43.25)(56.25,101.375)(84.75,43)
\qbezier(22,43.5)(48,71.625)(66,43.25)
\qbezier(22,43.25)(57,-15.75)(85,43.25)
\qbezier(44.5,43.25)(62.875,6.25)(84.75,43.25)
\put(22,43.25){\circle*{1.581}}
\put(95.5,42.75){\circle*{1.581}}
\put(105.75,43){\circle*{1.581}}
\put(117,43){\circle*{1.581}}
\put(139.75,42.75){\circle*{1.581}}
\put(163.5,42.5){\circle*{1.581}}
\put(180.5,42.75){\circle*{1.581}}
\put(45,43.25){\circle*{1.581}}
\put(65.75,43.5){\circle*{1.581}}
\put(84.75,43.75){\circle*{1.581}}\put(59.75,60.5){\makebox(0,0)[cc]{$\gamma_1$}}
\put(44.75,49.5){\makebox(0,0)[cc]{$\gamma_2$}}
\put(64,34.75){\makebox(0,0)[cc]{$\gamma_3$}}
\put(40.75,29.5){\makebox(0,0)[cc]{$\gamma_1$}}
\put(54.75,75){\makebox(0,0)[cc]{$e_1$}}
\put(44,60.5){\makebox(0,0)[cc]{$e_2$}}
\put(32.5,46.25){\makebox(0,0)[cc]{$e_2$}}
\put(55,45.75){\makebox(0,0)[cc]{$e_4$}}
\put(75,46){\makebox(0,0)[cc]{$e_5$}}
\put(61,23){\makebox(0,0)[cc]{$e_5$}}
\put(53.75,11.25){\makebox(0,0)[cc]{$e_1$}}
\put(94.5,43){\line(1,0){86.5}}
\qbezier(94.75,43)(132.375,140)(180.5,43)
\qbezier(180.5,43)(146.5,-41.625)(94.5,43.25)
\qbezier(105,42.75)(132.25,113.25)(180.5,42.75)
\qbezier(105.25,43)(136.5,86.5)(163.75,43)
\qbezier(116.75,43)(134.75,69.375)(163.75,42.25)
\qbezier(105.5,43)(145.875,-17.5)(179.75,43)
\qbezier(105.75,43)(130.875,12)(139.5,43)
\qbezier(139.5,43)(158.375,13.5)(179.75,43)
\put(132,83.75){\makebox(0,0)[cc]{$\gamma_1$}}
\put(135.5,70.75){\makebox(0,0)[cc]{$\gamma_3$}}
\put(130,58.5){\makebox(0,0)[cc]{$\gamma_7$}}
\put(135.75,48.25){\makebox(0,0)[cc]{$\gamma_8$}}
\put(159.75,36.25){\makebox(0,0)[cc]{$\gamma_3$}}
\put(124.75,36.5){\makebox(0,0)[cc]{$\gamma_7$}}
\put(143.25,23.5){\makebox(0,0)[cc]{$\gamma_3$}}
\put(136.5,6.5){\makebox(0,0)[cc]{$\gamma_1$}}
\put(138,94.5){\makebox(0,0)[cc]{$e_1$}}
\put(117.75,72.5){\makebox(0,0)[cc]{$e_5$}}
\put(99,45.5){\makebox(0,0)[cc]{$e_2$}}
\put(151.5,61.75){\makebox(0,0)[cc]{$e_{4}$}}
\put(111.75,46){\makebox(0,0)[cc]{$e_{12}$}}
\put(138.5,58.25){\makebox(0,0)[cc]{$e_{13}$}}
\put(127,45.5){\makebox(0,0)[cc]{$e_{13}$}}
\put(148.25,46){\makebox(0,0)[cc]{$e_{4}$}}\put(169.75,46){\makebox(0,0)[cc]{$e_{5}$}}
\put(124.75,31.25){\makebox(0,0)[cc]{$e_{4}$}}
\put(158.25,30.75){\makebox(0,0)[cc]{$e_{5}$}}
\put(143.75,17){\makebox(0,0)[cc]{$e_5$}}
\put(162.5,6.25){\makebox(0,0)[cc]{$e_{1}$}}
\end{picture}
\end{center}
\caption{The 2-core for $\la x_0, x_1x_2x_1\iv\ra$}
\label{f:x1}
\end{figure}

\nopagebreak[1000]

\begin{figure}[ht!]
\begin{center}
% This is a LaTeX picture output by TeXCAD.
% File name: [p53.pic].
% Version of TeXCAD: 4.3
% Reference / build: 30-Jun-2012 (rev. 105)
% For new versions, check: http://texcad.sf.net/
% Options on the following lines.
%\grade{\on}
%\emlines{\off}
%\epic{\off}
%\beziermacro{\on}
%\reduce{\on}
%\snapping{\off}
%\pvinsert{% Your \input, \def, etc. here}
%\quality{8.000}
%\graddiff{0.005}
%\snapasp{1}
%\zoom{4.0000}
\unitlength .8mm % = 2.845pt
\linethickness{0.4pt}
\ifx\plotpoint\undefined\newsavebox{\plotpoint}\fi % GNUPLOT compatibility
\begin{picture}(120.791,100)(0,0)
\put(33.5,46.75){\line(1,0){86.75}}
\qbezier(33.5,47)(76.125,117)(120.25,47)
\qbezier(120.25,47)(77,-31.25)(33.75,46.5)
\qbezier(45.75,46.75)(77.75,94.125)(119.75,47)
\qbezier(119.75,47)(82.625,-.375)(46,46.75)
\qbezier(46,46.75)(76.375,73.625)(101.25,47)
\qbezier(74.75,46.5)(95.25,22.875)(118.75,46.75)
\put(33.5,46.75){\circle*{1.581}}
\put(46.75,46.75){\circle*{1.581}}
\put(74.75,46.75){\circle*{1.581}}
\put(101.25,46.5){\circle*{1.581}}
\put(120,46.75){\circle*{1.581}}
\put(77.5,85.25){\makebox(0,0)[cc]{$f_1$}}
\put(78.75,73.25){\makebox(0,0)[cc]{$f_3$}}
\put(40.5,49.75){\makebox(0,0)[cc]{$f_2$}}\put(82.75,62){\makebox(0,0)[cc]{$f_4$}}
\put(109,49.5){\makebox(0,0)[cc]{$f_5$}}
\put(61.5,50){\makebox(0,0)[cc]{$f_6$}}
\put(86.5,49.75){\makebox(0,0)[cc]{$f_7$}}
\put(99,38.5){\makebox(0,0)[cc]{$f_8$}}
\put(79.25,27){\makebox(0,0)[cc]{$f_9$}}
\put(75.75,10.5){\makebox(0,0)[cc]{$f_{10}$}}
\put(55.75,67.5){\makebox(0,0)[cc]{$\delta_1$}}
\put(70.5,64.75){\makebox(0,0)[cc]{$\delta_2$}}
\put(73.5,52.5){\makebox(0,0)[cc]{$\delta_3$}}
\put(88.75,41.75){\makebox(0,0)[cc]{$\delta_4$}}
\put(63.5,38.75){\makebox(0,0)[cc]{$\delta_5$}}
\put(57.25,22.75){\makebox(0,0)[cc]{$\delta_6$}}
\end{picture}
\end{center}
\caption{The $2$-automaton for $x_1$}
\label{f:x23}
\end{figure}
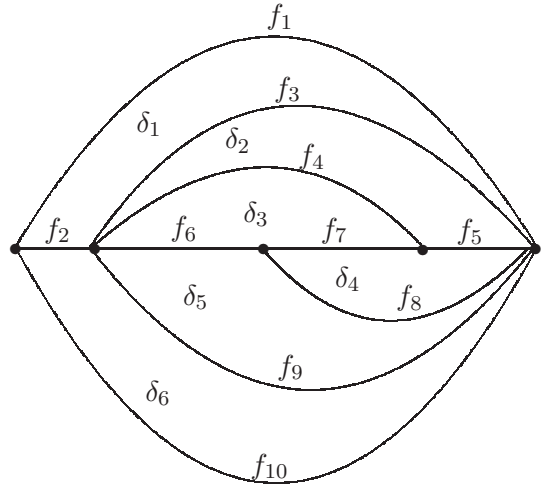

Now suppose that $x_1$ is accepted by $\La$. The diagram $\Delta$ for $x_1$ with labels of edges and cells is in Figure \ref{f:x23}. We should have a morphism $\psi$ from $\Delta$ to $\La$ sending $f_1$ and $f_{10}$ to $e_1$. Then $\psi(\delta_1)=\psi(\delta_6)=\gamma_1$ since $\La$ has only one cell with top edge $e_1$. This forces $\psi(f_2)=e_2$, $\psi(f_3)=\psi(f_9)=e_5$. Since $\La$ has only one positive cell with top edge $e_5$, we should have $\psi(\delta_2)=\gamma_3$. That means $\psi(f_4)=e_4, \psi(f_5)=e_5$. Again $\La$ has only one positive cell with top edge $e_4$. Therefore $\psi(\delta_3)=\gamma_7$, hence $\psi(f_6)=e_{12}, \psi(f_7)=e_{13}$. Now $\psi$ must map the positive cell $\delta_4$ to a cell with bottom edges $\psi(f_7)=e_{13}$ and $\psi(f_5)=e_5$. But $\La$ does not have such a cell, a contradiction. This contradiction shows, by Lemma \ref{l:au}, that $H$ does not contain $x_1$ and hence is a proper subgroup of $F$.
\endproof

\begin{Remark} One more way to show that $H=\la x_0, x_1x_2x_1\iv\ra$ is a proper subgroup of $F$ is to prove that $H$ does not act transitively on the set of finite binary fractions from $[0,1]$. That can be also done using the 2-core of $H$ (see \cite{Golan16}). One can ask whether a maximal subgroup of $F$ of infinite index can act transitively. The answer is ``yes", see \cite{Golan16}. That gives a strong negative answer to Savchuk's question.
\end{Remark}

\section{$F$ is quasi-residually finite}

Let $G$ be a finitely generated group. Recall that a subgroup $H < G$ is of \emph{quasi-finite} index if the interval $[H,G]$ in the lattice of subgroups of $G$ is finite, and  $G$ is \emph{quasi-residually finite} if it has a decreasing sequence of finitely generated subgroups $G>H_1>H_2...$ of quasi-finite index in $G$ with $\cap H_i=\{1\}$. In this section we prove

\begin{Theorem} \label{t:2} The R. Thompson group $F$ is quasi-residually finite.
\end{Theorem}

\proof Let $U$ be a finite set of finite binary fractions in $(0,1)$ %(i.e., numbers of the form $\frac{a}{2^b}$ for $b\in\mathbb{N}$ and $a<2^b$ a positive integer)
 and $H_U$ be the stabilizer of this set in $F$. It is clear that if $U_1\subset U_2\subset U_3...$ and $\cup U_i$ is the set of all finite binary fractions in $(0,1)$,
%dyadic rational numbers,
 then $F>H_{U_1}>H_{U_2}...$ and $\cap H_{U_i}=\{1\}$. It is also clear that each $H_U$ is isomorphic to the direct product of $|U|+1$ copies of $F$, hence is finitely generated. It remains to show that $H_U$ is of quasi-finite index in $F$. In fact we will prove more:

{\bf Claim 1.} The lattice $L_1$ of subgroups of $F$ containing $H_U$ is anti-isomorphic to the Boolean lattice $L_2$ of subsets of $U$, and the map $V\mapsto H_V$ for $V\subseteq U$ is an anti-isomorphism from $L_2$ to $L_1$.

We shall prove

{\bf Claim 2.} Let $g\in F\setminus H_U$. Then the subgroup $\la H_U, g\ra$ contains one of the subgroups $H_V$ where $V$ is a proper subset of $U$.

Note that Claim 2 implies Claim 1. Indeed, every subgroup $G$ of $F$ containing $H_U$ is generated by $H_U$ and a sequence of elements $g_1,g_2,\ldots$ such that $g_1\not\in H_U$ and for every $i$, $g_{i+1}\not\in \la H_U, g_1,\ldots, g_i\ra$. By Claim 2, $G$ is equal to $\la H_V, g_1,g_2,\ldots \ra$ where $V$ is a proper subset of $U$ and by induction on the size $|U|$ the subgroup $G$ is equal to $H_W$ for some $W\subseteq U$. Now it is obvious that if $V\ne W$, then $H_V\ne H_W$ and that $V\subseteq W$ if and only if $H_W\le H_V$. Hence the map $V\mapsto H_V$ from the lattice of subsets of $U$ to the lattice of subgroups of $F$ containing $H_U$ is a lattice anti-isomorphism.

Let us prove Claim 2. Let $u_1,...,u_n$ be the elements of $U$ in increasing order. Let us denote $u_0=0$. Since $g\not\in H_U$, it does not fix one of the numbers $u_j, j\ge 1$. Let $u_i$ be the smallest number in $U$ not fixed by $g$. Let $V=U\setminus \{u_i\}$. It is enough to show that $H_V\subseteq \la H_U, g\ra$. To do so, we will prove first that $\la H_U,g\ra$ contains some element of $H_V\setminus H_U$.

Let $U_g\subseteq U$ be the subset of elements of $U$ which are not fixed by $g$. In particular $u_i\in U_g$. Since for any $u\in U_g$, the elements in the orbit $g^n(u)$ for $n\in\mathbb{N}$ are all distinct, there is a power $k\in\mathbb{N}$ for which $g^k(U_g)\cap U=\emptyset$. By \cite[Lemma 4.2]{CFP} there is a function $h\in F$ which fixes all the points in $U\cup g^k(U_g\setminus\{u_i\})$ and does not fix $g^k(u_i)$. In particular $h\in H_U$.
Let $g_1=g^khg^{-k}$. Then $g_1$ fixes all $u\in U\setminus\{u_i\}$ (recall that composition in $F$ is from left to right). By construction, $g_1$ does not fix $u_i$. Therefore $g_1\in  \la H_U,g\ra\cap (H_V\setminus H_U)$.

To finish the proof, we will show that $\la H_U,g_1\ra=H_V$. Since $H_U$ and $g_1$ are contained in $H_V$ one inclusion is trivial. For the other direction, let $f\in H_V$. We can assume that $f$ does not belong to $H_U$, so in particular $f(u_i)\neq u_i$. Since $f$ is increasing, either $f(u_i)<u_i$ or $f\iv (u_i)<u_i$. Without loss of generality let $f(u_i)<u_i$. Similarly we can assume that $g_1(u_i)<u_i$. Lemma 4.2 of \cite{CFP} implies that the group $H_U$ acts transitively on the set of finite binary fractions from the interval $(u_{i-1}, u_i)$. Therefore there exists an element $h_1\in H_U$ such that $fh_1(u_i)=g_1(u_i)$. Therefore $p=fh_1g_1^{-1}$ fixes $u_i$. Since $p$ also belongs to $H_V$, we get that $p\in H_U$. Therefore $f=pg_1h_1\iv\in\la H_U,g_1\ra$.
\endproof

\begin{Remark} It is easy to see that the proof of Theorem \ref{t:2} can be adapted for every group $G$ of order preserving maps $Q\to Q$ where $Q$ is a dense subset of $\R$ provided $G$ is \emph{locally transitive} meaning that for every open interval $(a,b)$ in $\R$ the subgroup of all elements of $G$ fixing $Q\setminus (a,b)$ pointwise acts transitively on $(a,b)\cap Q$. Thus every such group $G$ is quasi-residually finite. That includes the group $\mathrm{Homeo}^+(\R)$ of all increasing homeomorphisms of $\R$.
\end{Remark}

Certainly not every finitely generated group is quasi-residually finite. The simplest (to explain) example is a torsion Tarski monster where every proper subgroup is finite and cyclic \cite{Olsh}. More generally every infinite group with the descending chain condition on subgroups is not quasi-residually finite.

For every non-residually finite group, it is interesting to know if it is quasi-residually finite. The next problem lists some of the most famous non-residually finite groups.

\begin{Problem} Is any of these groups quasi-residually finite:
\begin{itemize}
\item The free Burnside group $B_{m,n}$ for all sufficiently large exponents $n$ and rank $m\ge 2$,
\item The Higman group $H=\la a,b,c,d\mid b\iv ab=a^2, c\iv bc =b^2, d\iv cd =c^2, a\iv da =d^2\ra$,
\item The Baumslag-Solitar group $BS(2,3)=\la a,b\mid a^3=b\iv a^2b\ra$,
\item The Baumslag group $B=\la a,b,c\mid c\iv ac =b, b\iv ab =a^2\ra$,
\item The R. Thompson groups $T$ and $V$?
\end{itemize}
\end{Problem}

\begin{Remark} Recently Jordan Nikkel showed (unpublished) that the R. Thompson groups $T$ and $V$ are quasi-residually finite. \end{Remark}

By Theorems \ref{main} and Claim 1 in the proof of Theorem \ref{t:2}, if $H$ is Jones' subgroup $\arr F$ or one of Savchuk's subgroups $H_U$ for a finite set $U$ of binary fractions, then the interval $[H,F]$ in the lattice of subgroups of $F$ is  modular (in fact, distributive): the three-element chain in the first case and a Boolean lattice in the second case. If $H$ is a finite index subgroup of $F$, then it is normal in $F$ and $F/H$ is Abelian because $H$ contains the derived subgroup of $F$. Hence the lattice $[H,F]$ is also modular.

\begin{Problem} Is there a subgroup $H$ of $F$ of quasi-finite index such that the interval $[H,F]$ is not a modular lattice?
\end{Problem}

Finally notice that there are finitely generated and infinitely generated maximal subgroups of $F$ (say, $H_{\{\alpha\}}$ is finitely generated if $\alpha$ is a finite binary fraction, and infinitely generated if $\alpha$  is irrational \cite{Sav}). Since $F$ is finitely generated, every proper subgroup of $F$ is contained in a maximal subgroup of $F$.

\begin{Problem} Is it true that every finitely generated proper subgroup of $F$ is contained in a finitely generated maximal subgroup of $F$?
\end{Problem}

We cannot finish the paper about $F$ without asking an amenability question.

\begin{Problem} Is the Schreier graph of Jones' subgroup $\arr F$ amenable?
\end{Problem}

Note that if the answer is ``no'', then $F$ itself is not amenable.

\section{Some further results and open problems about subgroups of $F$}\label{5}

Since the paper was submitted several new results about subgroups of the Thompson group $F$ have been obtained by the authors. Here we list some of these results.

\subsection{Stabilizers of finite sets of points}  As noted above, if a finite set $U\subset (0,1)$ consists of finite binary fractions, then the subgroup $H_U$, the stabilizer of $U$ in $F$,  is isomorphic to the direct product of $|U|+1$ copies of $F$. It is mentioned in \cite{Sav1}, that if $U=\{\alpha\}$ where $\alpha$ is irrational, then $H_U$ is not finitely generated. In \cite{GS16} we clarify the algebraic structure of $F_U$ for arbitrary finite $U$. For every such $U=\{\alpha_1,\alpha_2,...,\alpha_n\}$ where $\alpha_1 < \alpha_2 <...< \alpha_n$ we  define its {\em type} $\tau(U)$, as the word of length $n$ in the alphabet $\{1,2,3\}$ as follows: for every $i$, the $i$th letter in $\tau(U)$ is $1$ if $\alpha_i$ is a finite binary fraction, $2$ if $\alpha_i$ is a rational but not a finite binary fraction, $3$ if $\alpha_i$ is irrational.

\begin{theorem}\label{th:a} (1) If $\tau(U)=\tau(V)$ then $H_U$ and $H_V$ are isomorphic.

(2) $H_U$ is finitely generated if and only if $U$ does not contain irrational numbers. In that case the minimal number of generators of $H_U$ is
$2k+m+2$ where $k$ is the number of $1$s in $\tau(U)$ and $m$ is the number of $2$s in $\tau(U)$.
\end{theorem}

In each case, we completely describe $H_U$ as an iterated HNN extension of a direct product of several copies of $F$, several copies of the derived subgroup $[F,F]$ of $F$ and several copies of the normal subgroup $L$ of $F$ of all functions with slope $1$ at $0$. In particular, if $\alpha$ is a rational number which is not a finite binary fraction, then $H_{\{\alpha\}}$ is isomorphic to the ascending HNN extension of $F\times F$ corresponding to the endomorphism $\phi$ defined as follows (it does not depend on $\alpha$ by Theorem \ref{th:a}):

$$\begin{array}{ll} \phi\colon & (x_0,\1)\to (x_0\oplus \1,\1), \\
& (x_1,\1)\to (x_1\oplus \1,\1),\\
& (\1,x_0)\to (\1,x_1),\\
& (\1,x_1)\to (\1,x_2)
\end{array}$$
where $\1$ is the identity function and the sum $f\oplus g$ of two functions $f,g$ in $F$ is defined as the function from $F$ which on $[0,\frac 12]$ coincides with the shrunk by the factor of 2 copy of $f$, i.e., $\frac12 f(2t)$, and on $[\frac 12,1]$ coincides with the shrunk by the factor of 2 copy of $g$ shifted by $\frac 12$, i.e., $\frac 12 g(2t-1)+\frac12$.

\subsection{Amenable maximal subgroups}

\begin{prob}  Does $F$ contain an elementary amenable maximal subgroup?
\end{prob}

The problem is open, but the following result was recently obtained by the first author.

\begin{theorem}[Golan \cite{Golan16}] \label{t:golan} There is a sequence of finitely generated subgroups $B<K<F$ such that $B$ is elementary amenable and maximal in $K$,  $K$ is normal in $F$ and $F/K$ is infinite cyclic.
\end{theorem}

Note that the subgroup $B$ is isomorphic to one of Brin's subgroups of $F$ from \cite{Brin}. Also note that $K$ contains the derived subgroup of $F$, hence it contains a copy of $F$. Thus $K$ is amenable if and only if $F$ is.

\subsection{The closure operation on the lattice of subgroups of $F$}\label{53}

The notion of 2-core of a subgroup of $F$ gives rise to the following natural notion.

\begin{definition} The {\em closure} $\Cl(H)$ of a subgroup $H$ of $F$ is the subgroup of $F$ consisting of all diagrams that are accepted by the 2-core $\La(H)$ of $H$.
\end{definition}

It is clear that all usual conditions of the closure operation are satisfied, that is, $H\le \Cl(H)$ (Lemma \ref{l:au}), $\Cl(\Cl(H))=\Cl(H)$ and if $H_1\le H_2$, then $\Cl(H_1)\le\Cl(H_2)$.  If $H$ is finitely generated, then the 2-automaton $\La(H)$ is finite, and so the membership problem in $\Cl(H)$ is decidable, hence the distortion function of $\Cl(H)$ is recursive. Recall that it is still not known \cite{GuSa99} whether the distortion function of every finitely generated subgroup of $F$ is recursive, and in fact no finitely generated subgroup with super-polynomial distortion function is known.

\begin{prob} Is every closed subgroup of $F$ quasi-isometrically embedded, i.e., its distortion function is linear?
\end{prob}

The closure operation preserves several properties of a subgroup. For example, it is not hard to see (and follows from Theorem \ref{t:golan1} below) that the closure of every non-trivial cyclic subgroup generated by a function that has no fixed points except $0$ and $1$ is cyclic. Moreover, using the description of solvable subgroups from \cite{Bleak} the first author proved

\begin{theorem} [Golan, \cite{Golan16}]\label{t:g5} If $H$ is finitely generated and solvable of class $k$, then $\Cl(H)$ is also finitely generated and solvable of class $k$.
\end{theorem}

Still the following problem is open.

\begin{prob} \label{prob4} Is it true that if $H$ is finitely generated, then $\Cl(H)$ is finitely generated?
\end{prob}

Note that in Theorem \ref{t:g5}, one cannot replace ``solvable'' by ``elementary amenable''. Indeed,  if $B$ is the group from Theorem \ref{t:golan}, then there is a copy of $B$ in $F$ whose closure contains a subgroup isomorphic to $F$.

\begin{definition} Let $h\in F$. Then {\em components} of $h$ are all elements of $F$ that coincide with $h$ on a closed interval $[a,b]$ with finite binary $a,b$ and are identity outside $[a,b]$ (in that case $h$ necessarily fixes $a$ and $b$).
\end{definition}

For example, if $h= f\oplus g$, then $f\oplus \1$ and $\1\oplus g$ are components of $h$. Clearly components of $h$ pairwise commute and $h$ is a product of its ``minimal'' components, i.e., components with connected support (the set of all non-fixed points of $[0,1]$).

It is easy to prove that if $h$ is in a subgroup $H$ of $F$, then the components of $h$ are in the closure of $H$. The converse was conjectured by Guba and the second author about 20 years ago, and was recently proved by the first author of this paper.

\begin{theorem}[Golan, \cite{Golan16}] \label{t:golan1} The subgroup $\Cl(H)$ coincides with the subgroup of $H$ generated by all components of elements of $H$. Moreover $\Cl(H)$ coincides with the subgroup of $F$ consisting of piece-wise linear functions $f$ (with finitely many pieces) such that every linear piece of $f$ is a restriction of a linear piece of some function from $H$.
\end{theorem}

This theorem immediately implies the following

\begin{cor} Let $X$ be any subset of the unit interval $[0,1]$. Then the stabilizer of $X$ in $F$ is a closed  subgroup of $F$, that is $\Cl(H)=H$. Thus all Savchuk's subgroups $H_U$ and Jones' subgroup $\arr F$ are closed.
\end{cor}

\begin{prob}\label{p:max} Is it true that every maximal subgroup of $F$ of infinite index is closed?
\end{prob}

Every closed subgroup $H$ of $F$ is a diagram group. Indeed, it is a diagram group of the 2-automaton $\La(H)$ viewed as a directed 2-complex. Moreover closed subgroups can be characterized ``abstractly'' as precisely the diagram groups $\DG(\kk,a)$ where $a$ is an edge and $\kk$ is a directed 2-complex where every positive cell has the form $t\to lr$ ($t,l,r$ are edges) and no foldings can be performed (which means that for every $t$ there is only one cell of the form $t\to lr$, and for every pair $(l,r)$ there is only one cell of the form $t\to lr$). Indeed, if $\kk$ is such a directed 2-complex, then mapping every edge to $x$ and every cell $t\to lr$ to $x\to x^2$ gives a morphism from $\kk$ to the Dunce hat (Figure \ref{f:1}) and it is not difficult to prove that the corresponding homomorphism from $\DG(\kk,a)$ to $F$ is injective and the image is closed (that follows from Theorem \ref{t:golan1}).

\begin{prob} \label{p:9} Suppose that $\Cl(H)=F$ and $H$ is not inside a proper subgroup of finite index of $F$. Is it true that $H=F$?
\end{prob}

The positive answer to Problem \ref{p:9} would give an easy algorithm to decide whether a finite set of elements of $F$ generates the whole $F$, i.e., solves the generation problem for $F$.
Indeed first we decide whether these elements are not inside any finite index subgroup of $F$. That can be done by looking at their images in $F/[F,F]$. Then construct the 2-core $\La(H)$ and check if it coincides with the 2-core $\La(F)$.

Note that in \cite{Golan16'}, the first author proved that the generation problem for $F$ is in fact decidable (the notion of 2-core is also used in the algorithm from \cite{Golan16'}).

Note also that a positive answer to Problem \ref{p:9} would imply a positive answer to Problem \ref{p:max}. Indeed, if $H\le F$ is a non-closed maximal subgroup of infinite index in $F$ then $\Cl(H)=F$ and $H$ is not contained in any proper finite index subgroup of $F$.

\end{document}